\documentclass[11pt,reqno]{amsart}
\usepackage{amsmath,dsfont}
\usepackage{amssymb}
\usepackage{amsthm}
\usepackage{color}
\usepackage{eucal}
\usepackage{tikz}
\usetikzlibrary{matrix,arrows,decorations.markings}
\usepackage{gastex}
\usepackage{amscd}
\usepackage{bm}

\usetikzlibrary{decorations.pathmorphing}
\DeclareSymbolFont{rsfscript}{OMS}{rsfs}{m}{n}
\DeclareSymbolFontAlphabet{\mathrsfs}{rsfscript}
\def\softd{{\leavevmode\setbox1=\hbox{d}\hbox
            to 1.15\wd1{d\kern-0.2ex{\char039}\hss}}}    
\def\softl{l\kern-0.3ex\raise0.1ex\hbox{'}\kern-0.3ex}   
\numberwithin{equation}{section}
\newtheorem{Thm}{Theorem}[section]
\newtheorem{Prop}[Thm]{Proposition}
\newtheorem{Lemma}[Thm]{Lemma}
\newtheorem{Def}{Definition}[section]

\newtheorem{Cor}[Thm]{Corollary}

\DeclareSymbolFont{rsfscript}{OMS}{rsfs}{m}{n}
\DeclareSymbolFontAlphabet{\mathrsfs}{rsfscript}

\theoremstyle{remark}
\newtheorem{Rmk}{Remark}
\newtheorem{Problem}{Problem}
\newtheorem{Notation}{Notation}[section]
\newcommand{\ord}{\operatorname{ord}}
\def\inv{^{-1}}
\def\wh#1{\widehat{#1}}
\def\til#1{\widetilde {#1}}
\def\ol#1{\overline{#1}}
\def\rb{{-\!\!\!-\!\!\!-\!\!\!-\!\!\!\!\!\rightarrow}}
\def\rbl{{-\!\!\!-\!\!\!-\!\!\!-\!\!\!-\!\!\!\!\!\rightarrow}}
\def\Ga{\Gamma}

\def\fF{\mathfrak{F}}

\def\vp{\varphi}

\def\cA{\mathrsfs A}

\def\cC{\mathrsfs C}
\def\cH{\mathrsfs H}
\def\cK{\mathrsfs K}

\def\mG{\mathcal G}
\def\mA{\mathcal A}

\def\mU{\mathcal U}
\def\mV{\mathcal V}
\def\sfN{\mathsf N}

\begin{document}

\title{{Formations}  of finite groups \linebreak with the M.~Hall property}
\author{K.~Auinger and A.~Bors}\thanks{The second author is supported by the Austrian Science Fund (FWF), project J4072-N32 ``Affine maps on finite groups'', and carried out part of the work on this paper while still enjoying the support of FWF project F5504-N26, a part of the Special Research Program ``Quasi-Monte Carlo Methods: Theory and Applications''.}
\address{Fakult\"at f\"ur Mathematik, Universit\"at Wien,
Oskar-Morgenstern-Platz 1, A-1090 Wien, Austria}
\email{karl.auinger@univie.ac.at}
\address{The University of Western Australia, Centre for the Mathematics of Symmetry and Computation,
35 Stirling Highway, 6009 Crawley, Australia}
\email{alexander.bors@uwa.edu.au}

\subjclass[2010]{20E18, 20F65, 05C25}
\keywords{profinite group, profinite graph, formation of finite groups}
\begin{abstract} The first examples of formations which are arboreous (and therefore Hall) but not freely indexed (and therefore not locally extensible) are found. Likewise, the first examples of solvable formations which are freely indexed and arboreous (and therefore Hall) but not locally extensible are constructed. Some open questions are also mentioned.
\end{abstract}
\maketitle
\section{Introduction}

This paper belongs to the aftermath of the Ribes--Zalesski\u\i-Theorem \cite{RZ}, which states that the set-product $H_1\cdots H_n$ of any finite number of finitely generated subgroups $H_i$ of a free group $F$ is closed in the profinite topology of $F$. The theorem originally was motivated by a paper by Pin and Reutenauer \cite{PR}: it provided a nice algorithm to compute the closure (with respect to the profinite topology) of a rational subset of  $F$ and implied the truth of the Rhodes type II conjecture, hence has attracted wide attention among semigroup theorists. Since then several papers were devoted to the study of which profinite topologies  on $F$ admit the Ribes--Zalesski\u\i-Theorem in the sense that the product of finitely generated closed subgroups should be closed again. The original proof by Ribes and Zalesski\u\i\ holds for the pro-$\mathbf V$-topology of any extension-closed variety of finite groups $\mathbf V$. This was generalized by Steinberg and the first author \cite{geometry0} to so-called arboreous varieties and to arboreous formations $\mathfrak{F}$ by the first author \cite{geometry1}. One intention of the present paper is to put this question into a more appropriate framework, namely profinite (Hausdorff) topologies in general: the restriction to pro-$\mathbf V$-topologies or pro-$\mathfrak{F}$-topologies for varieties $\mathbf{V}$ and/or formations $\mathfrak{F}$ is not really justified, and in section \ref{non-solvable} a profinite topology on $F$ is constructed which admits the Ribes--Zalesski\u\i-Theorem and such that the finite continuous quotients of $F$ are not closed under direct powers. Hence this topology cannot be the pro-$\mathfrak{F}$-topology of any formation $\mathfrak{F}$. As in  earlier papers, the geometry of the profinite Cayley graph of the (relative) profinite completion of $F$ plays an important role. It turns out that being ``tree-like'' of that graph is sufficient for the Ribes--Zalesski\u\i-Theorem to hold and is ``almost'' necessary (in a precisely defined sense, see Theorem \ref{RZTheorem}).

Up to now there was basically only one method known to produce profinite groups with tree-like Cayley graphs, namely via inverse systems which admit ``universal $S$-extensions'' (for $S$ a finite simple group, see subsection \ref{universal S extension}). The second intention of the paper is to present new ways to produce such groups which, in particular, avoid the use of the mentioned universal extensions. We shall present two different methods: one is by use of permutation groups (Section \ref{non-solvable}) the other is by use of a kind of modified universal $S$-extension (Section \ref{solvable}) for $S=C_p$ the cyclic group of prime order. Both constructions lead to examples that can be extended to the level of formations (that is, in both cases we get ``arboreous'' formations which are not ``locally extensible''). Two further properties of a profinite group and of a profinite Cayley graph --- that of being ``freely indexed'' and that of being ``Hall''  --- show up in this context and will be also discussed. Finally, in Section \ref{preliminaries} we collect all preliminaries while in Section \ref{geometry and topology} we present the relevant results concerning profinite topologies on a free group $F$ and its connection with the geometry of the Cayley graph of the (relative) profinite completion of $F$.

\section{Preliminaries}\label{preliminaries} In this section we collect all prerequisites needed for the paper; the main ingredients are: graphs and profinite graphs; categories of $A$-generated groups and $A$-labeled graphs; universal $S$-extensions of $A$-generated groups; subgroups of free groups.
\subsection{Graphs and profinite graphs}
We follow the Serre convention~\cite{Serre} and
define a \emph{graph} $\Gamma$ to consist of a set $V(\Gamma)$ of
\emph{vertices} and disjoint sets $E(\Gamma)$ of \emph{positively
oriented (or positive) edges} and $E\inv(\Gamma)$ of
\emph{negatively oriented (or negative) edges} together with
\emph{incidence} functions $\iota,\tau:E(\Gamma)\cup
E(\Gamma)\inv\rightarrow V(\Gamma)$ selecting the \emph{initial},
respectively, \emph{terminal} vertex of an edge $e$ and mutually
inverse bijections (both written: $e\mapsto e\inv$) between
$E(\Gamma)$ and $E\inv(\Gamma)$ such that $\iota e\inv = \tau e$
for all edges $e$ (whence $\tau e\inv =\iota e$, as well).  We set
$\til {E(\Gamma)}=E(\Gamma)\cup E\inv (\Gamma)$ and call it the
\emph{edge set} of $\Gamma$. Given this definition of a graph as a \emph{two-sorted structure} $(\til{E}, V,\iota,\tau,{}^{-1})$, the
notions of \emph{ subgraph (spanned by a set of edges)}, \emph{
morphism of graphs} and \emph{ projective limit of graphs} have  {well-defined} meanings. In particular, morphisms  respect sorts (edges to edges, vertices to vertices) and are compatible with the operations $\iota,\tau$ and ${}^{-1}$, subgraphs are closed under $\iota, \tau $ and ${}^{-1}$.

Edges are to be thought of geometrically: when
one draws an oriented graph, one draws only the positive edge $e$  and
thinks of $e^{-1}$ as being the same edge, but traversed in the
reverse direction. A \emph{geometric edge} in a graph is a pair $e^{\pm 1}=\{e,e^{-1}\}$ consisting of an edge and its inverse.
A \emph{path} $\pi$ in a graph $\Gamma$ is a finite sequence $\pi=e_1\dots
e_n$ of consecutive edges, that is $\tau e_i = \iota e_{i+1}$ for
all $i$; we define $\iota \pi= \iota e_1$ to be the initial vertex
of $\pi$ and $\tau \pi=\tau e_n$ to be the terminal vertex of $\pi$.  A path is \emph{reduced} if it does not contain a segment of
the form $ee\inv$ for any edge $e$. We also consider an empty path at each vertex.   A path $\pi=e_1\dots
e_n$ is a \emph{circuit at (base point)} $v$ or \emph{closed at} $v$ if $\iota \pi=v=\tau \pi$.
 A graph is \emph{connected} if any two vertices can be joined by a
path. A connected graph is a \emph{tree} if it does not contain a non-empty reduced circuit.

If $V(\Gamma)$, $E(\Gamma)$ and $E(\Ga)\inv$  are topological spaces,
$\til{E(\Gamma)}$ is the topological sum of $E(\Gamma)$ and
$E\inv(\Gamma)$, and $\iota$, $\tau$ and $(\ )\inv$ (in both directions) are continuous,
then $\Gamma$ is called a \emph{topological graph}. A
\emph{profinite graph} is a topological graph $\Gamma$ which is a
projective limit of finite, discrete graphs. It is well known
\cite{Mel, RZ3} that $\Gamma$ is profinite if and only if
$V(\Gamma)$ and $\til {E(\Gamma)}$ are both compact, totally
disconnected Hausdorff spaces.
Morphisms between profinite graphs,  {by definition}, are assumed to be
continuous. \emph{Subgraphs} of profinite graphs are understood in the category of profinite graphs: they must be closed as topological spaces. Moreover, a \emph{connected profinite graph} by definition is one all of whose finite continuous quotients are connected as abstract graphs (such are termed ``profinitely connected'' in \cite{AWred, AWsurvey}). For more information about profinite graphs the reader is referred to \cite{profinite graphs,AWred, AWsurvey, RZ3} and \cite[Section 2]{geometry0}.

Next we introduce geometric properties of profinite graphs that are central to the paper.
\begin{Def} \label{Hall-subgraph} Let $\Gamma$ be a connected profinite graph.
\begin{enumerate}
\item A connected subgraph $\Delta$ of $\Gamma$ is a \emph{Hall-subgraph} of $\Gamma$ if, whenever $\Delta$ contains the endpoints of a finite reduced path $\pi$ in $\Gamma$ then $\Delta$ contains (the graph spanned by) $\pi$ itself.
\item The graph $\Gamma$ has the \emph{Hall property} or simply is \emph{Hall} if every connected subgraph of $\Gamma$ is a Hall-subgraph.
\end{enumerate}
\end{Def}
A condition apparently  being stronger than being Hall  is the following.
\begin{Def} \label{definition:tree-like} A connected profinite graph $\Gamma$ is \emph{tree-like} if for every pair $\{u,v\}$ of vertices the intersection of all connected subgraphs of $\Gamma$ containing $\{u,v\}$ is connected.
\end{Def}
 A tree-like graph without loop edges is Hall. A finite graph which is Hall or tree-like necessarily is a tree (with possibly loop edges adjoined to some vertices in the tree-like case); so, these  concepts are two possibilities to transfer the notion of a tree to the context of profinite graphs. There are other possibilities  {as well}, all of which lead to different outcomes: pro-$p$-trees  and similarly defined graphs in the sense of \cite{RZ,RZ2,RZ3,profinite graphs} are tree-like and therefore Hall but the converse is not true.

The profinite graphs of primary interest are subgraphs of Cayley graphs of
 profinite groups where a \emph{profinite group} is
a compact, totally disconnected group, or, equivalently, a
projective limit of finite groups. We refer the reader
to~\cite{RZbook1} for basic definitions on
profinite groups and likewise for the definitions of a \emph{formation} $\mathfrak{F}$ and a \emph{variety} $\mathbf{V}$  \emph{of finite groups} and their \emph{relatively free} profinite groups.

\subsection{Categories of $A$-generated groups and $A$-labeled graphs} In this paper,
$A$ always denotes a fixed finite set (called \emph{alphabet} in this context) with at least two elements (called \emph{letters}). An \emph{$A$-generated group} is a pair $(G,\vp)$ where $G$ is a group and $\vp$ is a map $\vp:A\to G$ such that $G$ is generated by $\vp(A)$ ($\vp$ need not be injective). A \emph{morphism $\chi$ of $A$-generated groups} $(G,\vp)\to (H,\psi)$ is a morphism of groups $\chi:G\to H$ satisfying $\chi\circ\vp=\psi$; in particular, for any two $A$-generated groups $(G,\vp)$, $(H,\psi)$ there is at most one morphism $(G,\vp)\to (H,\psi)$ which is necessarily surjective and which we call the \emph{canonical morphism} and denote by $G\twoheadrightarrow H$ if it exists. Throughout we shall generously treat isomorphic $A$-generated groups as being equal; it follows that the collection of all $A$-generated groups forms a  {(partially ordered)} set (being in bijective correspondence with the set of all normal subgroups of the $A$-generated free group $F$). This set naturally forms a category (having at most one morphism between any two of its objects), and so does every subset. The category of all $A$-generated groups admits products (we need only finite products): given $A$-generated groups $(G,\vp)$ and $(H,\psi)$, their \emph{product}, denoted $(G\mathrel{\underset{A}{\times}}H, \vp\times \psi)$, is the subgroup of the Cartesian product $G\times H$ generated by the set $\{(\vp(a),\psi(a))\mid a\in A\}$ (which, as a group, is always a subdirect product of $G$ and $H$)  with  $\vp\times \psi:A\to G\mathrel{\underset{A}{\times}}H$, given by $a\mapsto (\vp(a),\psi(a))$. Every category of $A$-generated groups closed under finite products naturally is an {inverse system} of $A$-generated groups. Throughout the paper, an $A$-generated group $G$ is usually assumed to be finite (unless it is free or profinite, or stated otherwise).

Usually, the mapping $\vp:A\to G$ of an $A$-generated group is understood and will be omitted, and one thinks of $A$ being a subset of $G$. However, if $\vp$ happens to be not injective then it is often convenient to assume that $A$ is linearly ordered: $A=\{a_1<a_2<\cdots<a_n\}$, or equivalently, $A$ is the \textbf{sequence} $(a_1,\dots,a_n)$ rather than the set $\{a_1,\dots,a_n\}$. Then $\vp(A)$ becomes the \emph{generating sequence} $(\vp(a_1),\dots,\vp(a_n))$ of $G$ rather than the generating set $\{\vp(a_1),\dots,\vp(a_n)\}$. This distinction is relevant for example in the proofs of Theorems \ref{universal property} and \ref{quotients form formation} as well as in Subsection \ref{universal S extension}.

Next, we use $A\inv$ to denote a disjoint copy of
$A$ consisting of formal inverses $a\inv$ of the letters $a$ of $A$, and set  $\til A=A\cup A\inv$. The $A$-generated free group will throughout be denoted $F$, its elements will be realized as reduced words over the alphabet $\til A$. For an $A$ generated group $G=(G,\vp)$ and a word $w\in F$ we shall denote by $[w]_G$ the \emph{value of $w$ in $G$}, that is, $[w]_G:=\wh{\vp}(w)$ where $\wh{\vp}:F\twoheadrightarrow G$ is the unique extension of $\vp:A\to G$ to a morphism $F\to G$. Throughout, for every $A$-generated group $G$ and every $a\in \til A$ we shall write simply $a$ instead of $[a]_G$.

An \emph{$A$-labeled graph} is a graph together with a labeling function $\ell:\til E\to \til A$ such that $\ell(e)\in A$  and $\ell(e\inv)=\ell(e)\inv$ for each positive edge $e$. A \emph{morphism of labeled graphs} is assumed to respect the labelling. Given a path $\pi=e_1\dots e_n$ in a labeled graph, the label $\ell(\pi)$ of that path is just $\ell(e_1)\cdots\ell(e_n)$.

The \emph{Cayley graph of the $A$-generated group $(G,\vp)$},
denoted by $\Gamma(G)$, has vertex set $G$, edge set $G\times
\til A$, incidence functions given by $\iota (g,a) = g$, $\tau
(g,a) = g(\vp a)$ and involution $(g,a)\inv = (g(\vp a),a\inv)$. We call $a\in \til A$ the \emph{label} of $(g,a)$. The edge $(g,a)$ is usually drawn and thought of as $\underset{g}{\bullet}\!\overset{a}{\rbl}\!\!\!\!\!\underset{g(\vp a)}{\bullet}$.

 \subsection{The universal $S$-extension of an $A$-generated group}\label{universal S extension} The \emph{rank} $d(G)$ of a finitely generated (pro)finite group $G$ is the smallest possible size of a generating set (of a dense subgroup). For a finite simple group $S$ and a finitely generated (not necessarily $A$-generated) free group $R$ let $R(S)$ be the intersection of all normal subgroups $N$ of $R$ for which $R/N$ is a (finite) power of $S$. Then  $R/R(S)$ is a finite power of $S$ and $R(S)$ is characterisic in $R$. To see the former, let $r=d(R)$; it is well known that (i) every finite subdirect power of $S$ is actually a direct power of $S$ and (ii) there exists $k>1$ such that $S^k$ is $r$-generated, but $S^{k+1}$ is not. Let $K\unlhd R$ be such that $R/K\cong S^k$ and let $L\unlhd R$ be such that $R/L\cong S^l$ (for some $l\le k$).  Then $R/{K\cap L}$ is an $r$-generated subdirect product of $S^k$ and $S^l$ and, since this is a direct power of $S$ it must be isomorphic with $S^k$ (since it is $r$-generated and surjects to $S^k$). It follows that the projection $R/{K\cap L}\to R/K$ is injective, whence $K\cap L=K$, that is, $K\subseteq L$.
 Altogether, every normal subgroup $L$ of $R$ for which $R/L$ is a direct power of $S$ contains  $K$. It follows that $K=R(S)$. The interested reader may consult \cite[Theorem 2.22]{Collins} for more details. To see the latter, let $\alpha:R\to R$ be an automorphism; for $L\unlhd R$ with $R/L$ a power of $S$ then $R/L\cong \alpha(R)/\alpha(L)=R/\alpha(L)$ so that $\alpha(L)$ also belongs to the set of normal subgroups the intersection of which is $R(S)$, hence this intersection is  invariant under $\alpha$.

For an $A$-generated finite group $G$ with canonical morphism $\varphi:F\twoheadrightarrow G$ let $R:=\mathrm{ker}(\varphi)$. We set
$$G^{A,S}:=F/R(S).$$
Thus $G^{A,S}$ is an $A$-generated group, and an extension of $R/R(S)$ by $G$; immediately from the definition we get that $G^{A,S}$ enjoys the following universal property: suppose that $H$ is an $A$-generated group such that $H\overset{\psi}{\twoheadrightarrow}G$ and $\mathrm{ker}(\psi)$ is a power of $S$, then there exists a (unique) morphism of $A$-generated groups $\sigma: G^{A,S}\twoheadrightarrow H$.
Hence we call $G^{A,S}$ the \emph{$A$-universal $S$-extension} of $G$ \cite{geometry1}. We note that
$$\ker(G^{A,S}\twoheadrightarrow G)=\ker(F/R(S)\twoheadrightarrow F/R)=R/R(S)$$ and hence
\begin{equation}\label{Schreier for universal extension}
d(\ker(G^{A,S}\twoheadrightarrow G))=d(R/R(S))=d(R)=\vert G\vert(\vert A\vert -1)+1
\end{equation}
since $\vert G\vert =[F:R]$ and so (by the Schreier formula) the rank of $G^{A,S}$ is $\vert A\vert$, no matter what the rank of $G$ actually is. It should also be noted that, as a mere group,  $G^{A,S}$ depends on $G$, $S$ and $\vert A\vert$, but not on $A$. That is, it does \textbf{not} depend on the way the generating sequence $A$ is chosen within $G$, or, in other words, how the canonical morphism $\varphi:F\twoheadrightarrow G$ is actually chosen. (That it does depend has been erroneously claimed in several papers \cite{mytype2, geometry1, geometry0, constructive, supersolvable, F-inverse}.)

Indeed, let $\vp:F\twoheadrightarrow F/R=G$ and $\vp':F\twoheadrightarrow F/R(S)=G^{A,S}$ and $\widehat{\vp}: G^{A,S}\twoheadrightarrow G$ be the $A$-canonical morphisms, in particular $\vp=\wh{\vp}\circ \vp'$. Suppose that $B$ is an alphabet of size $\vert A\vert$, and $\psi:B\to G$ is a mapping making $G$ a $B$-generated group; we can consider the $B$-universal $S$-extension $G^{B,S}$ of $G$. That is, the group $G^{B,S}$ is, subject to a suitable mapping $\psi':B\to G^{B,S}$ a $B$-generated group with $B$-canonical morphism $\wh{\psi}:G^{B,S}\twoheadrightarrow G$. Since $\vert A\vert =\vert B\vert$ we can lift (with respect to $\wh{\psi}$) the generating sequence $A$ of $G$ to a generating sequence of $G^{B,S}$ \cite{Ga:lifting generators}. More precisely, there exists a map $\vp_1:A\to G^{B,S}$ such that:
\begin{itemize}
\item $\vp_1(A)$ generates $G^{B,S}$
\item $\wh{\psi}\circ \vp_1=\vp\vert_A$.
\end{itemize}
This way, $G^{B,S}$ becomes \textbf{some} $A$-generated extension of $G$ with canonical morphism $\wh{\psi}$ the kernel of which is a power of $S$. The universal property of $(G^{A,S}, \wh{\vp})$ then implies that $\wh{\vp}$ factors through $G^{B,S}$, in particular, there exists a surjective morphism $G^{A,S}\to G^{B,S}$. By symmetry there is also a surjective morphism $G^{B,S}\to G^{A,S}$ so that finiteness implies $G^{A,S}\cong G^{B,S}$.

For $S=C_p$ (the cyclic group of prime order $p$) we set $G^{A,\mathbb{Z}/p}:=G^{A,C_p}$  (or simply $G^{\mathbb{Z}/p}$ if the generating set $A$ of $G$ is clear). In this case $R(S)=R(C_p)=R^p[R,R]$ and $G^{A,\mathbb{Z}/p}$ is called the \emph{$A$-universal Gasch\"utz $p$-extension} or simply \emph{Gasch\"utz extension} \cite{Gaschuetz}. Moreover, a concrete model of $G^{A,\mathbb{Z}/p}$ can be given as follows. Let $E$ be the set of positive edges of the Cayley graph $\Gamma(G)$ of $G$; $G$ acts on $E$ by left multiplication: if we identify $E$ with $G\times A$ then this action can be most naturally described as:
\begin{equation}\label{action}
{}^g(h,a)=(gh,a)\mbox{ for all }g\in G\mbox{ and }(h,a)\in E.\end{equation}
Let $\mathbb{F}_p[E]$ be the additive group of the $\mathbb{F}_p$-vector space with basis $E$; the action (\ref{action}) extends to an action of $G$ on $\mathbb{F}_p[E]$ by automorphisms, hence we can form the semidirect product $\mathbb{F}_p[E]\rtimes G$ with respect to that action. Then, as it has been mentioned in \cite{geometry0, mytype2, constructive, supersolvable, Ballester},  $G^{A,\mathbb{Z}/p}$ is isomorphic with an $A$-generated subgroup of $\mathbb{F}_p[E]\rtimes G$, namely
$$G^{A,\mathbb{Z}/p}\cong\langle(e_a,a)\mid a\in A\rangle\le \mathbb{F}_p[E]\rtimes G$$
for $e_a=(1,a)$.

\subsection{Subgroups of free groups and Stallings automata}\label{stall}
Again let $F$ denote the free $A$-generated group. It is well known that finitely generated subgroups $H$ of $F$ can be encoded in terms of finite, labeled, pointed graphs \cite{kapmas,MSW,Stallings}. Let $\cA$ be a finite $A$-labeled graph; $\cA$ is \emph{folded} or an \emph{inverse automaton} if for every letter $a\in A$ and every vertex $v$ there exists at most one edge starting and at most one edge ending at $v$ and having label $a$. In a folded graph (inverse automaton), the letters from $\til A$ induce partial injective mappings on the vertex set (usually considered to act on the right): $v\mapsto v\cdot a$ for vertices $v$ and $a\in \til A$; in particular, for every vertex $v$ and every word $w$ in the letters of $\til A$ there is at most one path starting at $v$ and having label $w$. From now on we assume all graphs to be folded. Suppose that $\cA$ has a distinguished vertex $1_\cA$ (the base point) and let $L(\cA,1_\cA)$ be the set of all elements $w\in F$ ($w$ given as a reduced word in $\til A$) such that $w$ labels a closed path at $1_\cA$ in $\cA$. Then $L(\cA,1_\cA)$ is a finitely generated subgroup of $F$.

A  graph $\cA$ with base point $1_\cA$ is \emph{reduced} if no vertex except perhaps $1_\cA$ has degree one (where the degree of a vertex $v$ is the number of positive edges $e$ for which $\iota e =v$ or $\tau e =v$).  {Then}, for every finitely generated subgroup $H$ of $F$ there is up to isomorphism exactly one finite, connected, $A$-labeled, reduced graph $\cA$ with base point $1_\cA$ such that $L(\cA,1_\cA)=H$. We shall denote this graph $\cH$ with base point $1_\cH$. It can be obtained as follows.  The \emph{Schreier graph} $\Sigma(F,H,A)$   has vertex set the set $H{\setminus}F$ of all right cosets $Hg$ in $F$ with respect to $H$ and edges $\underset{Hg}{\bullet}\!\overset{a}{\rb}\!\underset{Hga}{\bullet}$ for $g\in F$ and $a\in \til A$; the edge set then can be identified with $H{\setminus}F\times \til A$. Let the \emph{core graph} $\mathrm{core}(\Sigma(F,H,A),H)$ of $\Sigma(F,H,A)$ with respect to the base point $H$ be the subgraph of $\Sigma(F,H,A)$ spanned by all edges which are contained in a reduced closed path at base-vertex $H$. Then $\mathrm{core}(\Sigma(F,H,A),H)$ is isomorphic to $\cH$ and will be the called the \emph{core graph} or the \emph{Stallings automaton} of $H$. For a more constructive method to find $\cH$ the reader is referred to \cite{kapmas,MSW, Stallings}. 

A finite $A$-labeled graph $\cA$ is \emph{complete} or a \emph{permutation automaton} if the degree of every vertex is $2\vert A\vert$. That is, every letter $a\in \til A$ induces a permutation of the set of vertices of $\cA$. The group generated by these permutations is the \emph{transition group} $T_{\cA}$ of $\cA$ which is an $A$-generated group. Since every permutation representation of a group $G$ is equivalent to a permutation representation on the set of right cosets $U{\setminus}G$ for some subgroup $U$ of $G$, every finite permutation automaton is isomorphic with some finite Schreier graph. Indeed, if we fix a vertex $1_\cA$ in the complete graph $\cA$ then from   \cite[Section 1.6]{Robinson} it follows that $(\cA,1_\cA)$ and $(\Sigma(T_\cA,U,A),U)$ are isomorphic as pointed graphs where $U\le T_\cA$ is the stabilizer of $1_\cA$. Any permutation automaton $\cC$ in which an incomplete inverse automaton $\cA$ can be embedded will be called a \emph{completion of $\cA$}; in this case, we also say that $\cC$ \emph{extends} $\cA$.

\section{Profinite topologies and geometry of profinite graphs}\label{geometry and topology}
\subsection{Profinite topologies on the free group} Again we let  $F$ be the free group on some fixed finite set $A$ with $\vert A\vert\ge 2$.

\begin{Def}\label{Hausdorff filtered}
A family $\sfN$  of normal subgroups of $F$ of finite index is \emph{Hausdorff filtered} if
\begin{enumerate}\label{fisubgroups}
\item $K,N\in \sfN\Rightarrow K\cap N\in \sfN$
\item $K\in \sfN \Rightarrow N\in \sfN$ for all $N\unlhd F$ for which $K\subseteq N$
\item $\bigcap \sfN=\{1\}$.
\end{enumerate}
\end{Def}
For a Hausdorff filtered family $\sfN$ the set of all quotients $\mathfrak{N}:=\{F/N\mid N\in \sfN\}$ is  a category of finite $A$-generated groups. Conditions (1)--(3) imposed on the family $\sfN$ are reflected in the behavior of the category $\mathfrak{N}$:
\begin{Prop}\label{category of quotients}
A category $\mathfrak{N}$ of finite $A$-generated groups is comprised  of the quotients of some Hausdorff filtered family $\sfN$ if and only if
\begin{enumerate}
\item $\mathfrak{N}$ is closed under finite products of $A$-generated groups
\item $\mathfrak{N}$ is closed under quotients
\item $\mathfrak{N}$ does not satisfy any nontrivial relation, that is, for every $w\in F$, $w\ne 1$, there exists a group $G\in\mathfrak{N}$ for which $[w]_G\ne 1$.
 \end{enumerate}
\end{Prop}
We could equally well have started with a category $\mathfrak{N}$ of finite $A$-generated groups satisfying conditions (1-3) of Proposition \ref{category of quotients}, the corresponding Hausdorff filtered family $\sfN$ of finite index normal subgroups of $F$ then would be the family $\{\ker(F\twoheadrightarrow G)\mid G\in \mathfrak{N}\}$.
\begin{Def}\label{relation-free PQ category}
A category $\mathfrak{N}$ of finite $A$-generated groups that satisfies conditions (1) and (2) of Proposition \ref{category of quotients} is called \emph{$PQ$-category}; one that satisfies  condition (3) is called \emph{relation-free}.
\end{Def}
The definition of a $PQ$-category is very reminiscent of that of a formation; indeed, for every formation $\mathfrak{F}$ the category $\mathfrak{F}_A$ of its $A$-generated members is clearly a $PQ$-category; however, the converse need not be true: given a $PQ$-category $\mathfrak{N}$, there is no reason why $\mathfrak{N}$ should be comprised of \textbf{all} $A$-generated members of some formation (see section \ref{non-solvable} for examples) --- this is somehow in contrast with Remark 3.1.1 in \cite{RZbook1}.

The family $\sfN$, or equivalently the category $\mathfrak{N}$, defines a profinite topology (actually a uniformity) on $F$ --- called the \emph{pro-$\mathfrak{N}$-topology} on $F$ --- by letting $\sfN$ be a basis for the neighborhoods of $1$. We note that a subgroup $H$ of $F$ is open in that topology if and only if $N\subseteq H$ for some $N\in \sfN$, which is the case if and only if the core $H_F$ belongs to $\sfN$. Moreover, $H$ is closed if and only if $H$ is the intersection of all open subgroups $U$ containing $H$ \cite[Theorem 3.3]{Hall}. The pro-$\mathfrak{N}$ topology  is the weakest topology  making  {continuous} all canonical morphisms from $F$ to the members of $\mathfrak{N}$ (considered as discrete topological  spaces).
%

For a $PQ$-category $\mathfrak{N}$, the \emph{pro-$\mathfrak{N}$-completion} of $F$ is the profinite group
$$\wh{F_\mathfrak{N}}:=\varprojlim_{N\in \sfN}F/N=\varprojlim_{G\in \mathfrak{N}}G,$$
which is an $A$-generated profinite group (where an \emph{$A$-generated profinite group $\mG$} is a pair $(\mG,\vp)$ where $\vp:A\to\mG$ is a map such that $\vp(A)$ generates $\mG$ as a profinite group, that is, the abstract subgroup $\left<\vp(A)\right>$ generated by $\vp(A)$ is dense in $\mG$). Recall that the category $\mathfrak{N}$ is naturally an inverse system of finite groups. The condition $\bigcap \sfN=\{1\}$ implies that the mapping $F\to \wh{F_\mathfrak{N}}$, $w\mapsto (wN)_{N\in \sfN}$ embeds $F$ in $\wh{F_\mathfrak{N}}$. In particular, the abstract subgroup of $\wh{F_\mathfrak{N}}$ generated by $A$ is $F$. Conversely, every $A$-generated profinite group $\mG$ for which the abstract subgroup generated by $\vp(A)$ is $F$ is of the form $\wh{F_\mathfrak{N}}$ for $\mathfrak{N}$ a relation-free $PQ$-category.

Given a  morphism of finite $A$-generated groups $\vp:G\twoheadrightarrow H$ we have a morphism $\Ga(G)\twoheadrightarrow \Ga(H)$, usually also denoted $\varphi$, of finite $A$-labeled graphs which maps each vertex $g$ to $\vp g$ and each edge $(g,a)$ to $(\vp g,a)$. The category $\mathfrak{N}$ of finite $A$-generated groups thus leads to the category $\Gamma(\mathfrak{\mathfrak{N}})$ of finite $A$-labeled (Cayley) graphs which is an inverse system of finite $A$-labeled (Cayley) graphs. Setting $$\Gamma(\wh{F_\mathfrak{N}}):= \varprojlim_{G\in \mathfrak{N}}\Gamma(G)$$ we get the Cayley graph $\Gamma(\wh{F_\mathfrak{N}})$ of $\wh{F_\mathfrak{N}}$, which is a connected profinite graph with vertex set $\wh{F_\mathfrak{N}}$ and edge set $\wh{F_\mathfrak{N}}\times \til A$. The geometry of this graph encodes properties of the pro-$\mathfrak{N}$-topology of $F$ and is the main topic of the present paper.

\subsection{Dissolving constellations: how to get tree-like \hyphenation{Cayley} Cayley graphs} The purpose of this subsection is to characterize the $PQ$-categories $\mathfrak{N}$ for which $\Gamma(\wh{F_\mathfrak{N}})$ is either tree-like or Hall (recall Definitions \ref{Hall-subgraph} and \ref{definition:tree-like}).

\begin{Def}\label{constellation}
Let $\Ga$ be a connected (pro)finite graph with distinguished vertex $1$; a \emph{constellation} in $\Ga$ is a triple $(\Xi,g,\Theta)$ where
\begin{enumerate}
\item $g$ is a vertex of $\Ga$, $\Xi$ and $\Theta$ are connected subgraphs of $\Ga$,
\item $1,g\in \Xi\cap \Theta$,
\item $1$ and $g$ are in distinct connected components of $\Xi\cap \Theta$.
\end{enumerate}
\end{Def}
We note that a connected profinite graph is tree-like if and only if it does not admit a constellation \cite[Proposition 2.6]{geometry0}. The following definitions are crucial.
\begin{Def}\label{dissolve constellation}
Let $G$ be a finite $A$-generated group and $(\Xi,g,\Theta)$ be a constellation in $\Gamma(G)$; an $A$-generated group $H$ \emph{dissolves} the constellation $(\Xi,g,\Theta)$ if $[u]_H\ne [v]_H$ for all pairs of words $(u,v)\in F\times F$ for which
\begin{enumerate}
\item $[u]_G=g=[v]_G$
\item the path in $\Gamma(G)$ labeled $u$ starting at $1$ runs entirely in $\Xi$, that labeled $v$ (also starting at $1$) runs entirely in $\Theta$.
\end{enumerate}
\end{Def}
The next notion of (pre)dissolver turns out to be essential for the understanding of $PQ$-categories $\mathfrak{N}$ for which the Cayley graph of the profinite completion $\wh{F_\mathfrak{N}}$ is tree-like.
\begin{Def}\label{dissolver} Let $G$ be a finite $A$-generated group.
\begin{enumerate}
\item an $A$-generated group $H$ is a \emph{predissolver of} $G$ if $H$ dissolves every constellation of $\Gamma(G)$;
\item a predissolver $H$ of $G$ which in addition satisfies $H\twoheadrightarrow G$ is a \emph{dissolver of} $G$.
\end{enumerate}
\end{Def}
For every predissolver $H$ of $G$, the product $H\mathrel{\underset{A}{\times}}G$ is a dissolver of $G$. By \cite[Proposition 4.1]{geometry1} we have
\begin{Thm}\label{Cayley:tree-like} Let $\mathfrak{N}$ be a $PQ$-category; then the Cayley graph $\Gamma(\wh{F_\mathfrak{N}})$ is tree-like if and only if every $G\in\mathfrak{N}$ admits a (pre)dissolver in $\mathfrak{N}$.
\end{Thm}

We turn to the Hall property. We first recall from \cite[Proposition 3.10]{supersolvable} the following necessary and sufficient condition on an $A$-generated profinite group to have a Cayley graph with the Hall property.
\begin{Thm} \label{CayleyHallsupersol} The Cayley graph $\Gamma(\wh{F_\mathfrak{N}})$ is Hall if and only if for each open subgroup $\mU$ of $\wh{F_\mathfrak{N}}$ and each edge $e$ of $\Gamma(\wh{F_\mathfrak{N}})$ the graph $\Gamma(\wh{F_\mathfrak{N}})\setminus \mU e^{\pm 1}$ is disconnected.
\end{Thm}
Suppose that $e=(g,a)$ for $g\in \wh{F_\mathfrak{N}}$ and $a\in \til A$; then $\mU e^{\pm 1}=\mU g(1,a)^{\pm 1}$ and
$$\Gamma(\wh{F_\mathfrak{N}})\setminus \mU e^{\pm 1}\cong g^{-1}\big(\Gamma(\wh{F_\mathfrak{N}})\setminus \mU g(1,a)^{\pm 1}\big)= \Gamma(\wh{F_\mathfrak{N}})\setminus g^{-1}\mU g(1,a)^{\pm 1}.$$
Since $\mU$ is open if and only if $g^{-1}\mU g$ is open, in Theorem \ref{CayleyHallsupersol} we may restrict ourselves to edges of the form $e=(1,a)$ where $a\in \til A$. Moreover, if $\mU\subseteq \mV$ are open subgroups and $\Gamma(\wh{F_\mathfrak{N}})\setminus \mU (1,a)^{\pm 1}$  is disconnected then so is $\Gamma(\wh{F_\mathfrak{N}})\setminus \mV (1,a)^{\pm 1}$; from this it follows that in Theorem \ref{CayleyHallsupersol} it suffices to consider open  {\emph{normal}} subgroups $\mU$. Let $\mU$ be an open normal subgroup of $\wh{F_\mathfrak{N}}$ for which $\Gamma(\wh{F_\mathfrak{N}})\setminus \mU (1,a)^{\pm 1}$ is disconnected  {and let $G:=\wh{F_\mathfrak{N}}/\mU$. There exists a finite quotient $H$ of $\wh{F_\mathfrak{N}}$ with $\wh{F_\mathfrak{N}}\overset{\vp}{\twoheadrightarrow}H\overset{\psi}{\twoheadrightarrow}G$ and such that
$$\vp(\Ga(\wh{F_\mathfrak{N}})\setminus \mU(1,a)^{\pm 1})=\Ga(H)\setminus N(1,a)^{\pm 1}$$
is disconnected for $N=\ker\psi$. Altogether we may modify Theorem \ref{CayleyHallsupersol} as follows.
\begin{Thm}\label{modifiedCayleyHall} The Cayley graph $\Gamma(\wh{F_\mathfrak{N}})$ is Hall if and only if each $G\in \mathfrak{N}$ admits $H\in \mathfrak{N}$, $H\overset{\psi}\twoheadrightarrow G$ such that for each $a\in \til A$ and $N=\ker \psi$ the graph $\Ga(H)\setminus N(1,a)^{\pm 1}$ is disconnected.
\end{Thm}
}

Next, we are going to analyze graphs of the form $\Gamma(H)\setminus N(1,a)^{\pm 1}$ where $H\in\mathfrak{N}$ and $N$ is some normal subgroup of $H$. For a finite group $G$ with Cayley graph $\Gamma(G)$ and $a\in \til A$ we consider the constellation
$$\Delta_a:=(\Gamma(G)\setminus (1,a)^{\pm 1}, a, \{1,(1,a)^{\pm 1},a\}).$$
\begin{Thm}\label{disconnecting} Let $a\in \til A$ and $\vp:H\twoheadrightarrow G$ with $N=\ker \vp$; then the following assertions are equivalent:
\begin{enumerate}
\item $\Gamma(H)\setminus N(1,a)^{\pm 1}$ is disconnected
\item $1$ and $a$ are in distinct connected components of $\Gamma(H)\setminus N(1,a)^{\pm 1}$
\item for all $n\in N$, $n$ and $na$ are in distinct connected components of $\Gamma(H)\setminus N(1,a)^{\pm 1}$
\item $H$ dissolves the constellation $\Delta_a$.
\end{enumerate}
\end{Thm}
\begin{proof} The implications (3) $\Rightarrow$ (2) $\Rightarrow$ (1) are trivial. In order to show (4) $\Rightarrow$ (3) let $n\in N$ and suppose that there is a word $w\in F$ labeling a path $n\to na$ which runs in $\Gamma(H)\setminus N(1,a)^{\pm 1}$. Then $[w]_H=a$, and $w$ also labels a path $1\to a$ which runs in
$$n^{-1}(\Gamma(H)\setminus N(1,a)^{\pm 1})=n^{-1}\Gamma(H)\setminus n^{-1}N(1,a)^{\pm 1} = \Gamma(H)\setminus N(1,a)^{\pm 1}.$$
The projection under $\vp$ of that path is a path $1\to a$ in $\Gamma(G)$ which runs in $\Gamma(G)\setminus(1,a)^{\pm 1}$. However, since $H$ dissolves $\Delta_a$ this is not possible, leading to a contradiction.

We are left with showing the implication (1) $\Rightarrow$ (4). We show that if $H$ does not dissolve $\Delta_a$ then $\Gamma(H)\setminus N(1,a)^{\pm 1}$ is connected. So, suppose that $H$ does not dissolve $\Delta_a$. There exists a word $w\in F$ which labels a path $1\to a$ which runs inside $\Gamma(G)\setminus (1,a)^{\pm 1}$ and such that $[w]_H=a$. In particular, $w$ labels a path $1\to a$ in $\Gamma(H)$ and that path does not traverse any edge of $N(1,a)^{\pm 1}$ because its projection to $\Gamma(G)$ avoids the edge $(1,a)^{\pm 1}$. In particular, there is a path $\pi_a:1\to a$ in $\Gamma(H)\setminus N(1,a)^{\pm 1}$. Let $n\in N$; the shifted path $n\pi_a$ runs from $n$ to $na$ and runs inside $n(\Gamma(H)\setminus N(1,a)^{\pm 1})=\Gamma(H)\setminus N(1,a)^{\pm 1}$. Let now $u,v$ be any two vertices of $\Gamma(H)$ and let $\pi$ be any path $u\to v$. Whenever $\pi$ traverses an edge $(n,na)^{\pm 1}\in N(1,a)^{\pm 1}$ then replace it by the path $(n\pi_a)^{\pm 1}$. This yields a path from $u$ to $v$ which runs entirely in $\Gamma(H)\setminus N(1,a)^{\pm 1}$. Altogether, $\Gamma(H)\setminus N(1,a)^{\pm 1}$ is connected, as requested.
\end{proof}

We are motivated to modify Definition \ref{dissolver}.
\begin{Def}\label{weak dissolver} Let $G$ be a finite $A$-generated group.
\begin{enumerate}
\item an $A$-generated group $H$ is a \emph{weak predissolver of} $G$ if $H$ dissolves the constellation $\Delta_a$ of $\Gamma(G)$ for every $a\in \til A$;
\item a weak predissolver $H$ of $G$ which in addition satisfies $H\twoheadrightarrow G$ is a \emph{weak dissolver of $G$}.
\end{enumerate}
\end{Def}
Again, if $H$ is a weak predissolver of $G$ then $H\mathrel{\underset{A}{\times}}G$ is a weak dissolver. We now are able to formulate an analogue of Theorem \ref{Cayley:tree-like};  {it is an immediate consequence of Theorems \ref{modifiedCayleyHall} and \ref{disconnecting}.}
\begin{Thm}\label{Cayley:Hall} Let $\mathfrak{N}$ be a $PQ$-category; then the Cayley graph $\Gamma(\wh{F_\mathfrak{N}})$ is Hall if and only if every $G\in\mathfrak{N}$ has a weak (pre)dissolver in $\mathfrak{N}$.
\end{Thm}

In \cite[Theorem 3.12]{supersolvable} an alternative characterization  {is given of when $\wh{F_\mathfrak{N}}$ has a Hall Cayley graph, in terms of irredundant generating sets of the open subgroups of $\wh{F_\mathfrak{N}}$.} From this, a seemingly unrelated property, introduced by Lubotzky and van den Dries \cite{LubotzkyDries} for finitely generated profinite groups, came into play which turned out to be a sufficient condition for the Hall property.
\begin{Def}\label{def:freely indexed}
The profinite group $\wh{F_\mathfrak{N}}$ is \emph{$A$-freely indexed} if for each open subgroup $\mU$ the rank $d(\mU)$ is given by the Schreier formula:
\begin{equation}\label{Schreier formula}
d(\mU)=[\wh{F_\mathfrak{N}}:\mU](\vert A\vert -1)+1.
\end{equation}
\end{Def}
Note that for $\mU=\wh{F_\mathfrak{N}}$ one immediately has that $d(\wh{F_\mathfrak{N}})=\vert A\vert$.
The following has been shown in \cite[Corollary 3.13]{supersolvable}.
\begin{Prop} The Cayley graph $\Gamma(\wh{F_\mathfrak{N}})$ of every freely indexed profinite group $\wh{F_\mathfrak{N}}$ is Hall.
\end{Prop}

Whether $\wh{F_\mathfrak{N}}$ is freely indexed can be easily expressed in terms of the corresponding Hausdorff filtered family $\sfN$ and likewise in terms of the associated relation-free $PQ$-category $\mathfrak{N}$.
\begin{Prop}\label{char:freely indexed} Let $\sfN$ be a Hausdorff filtered family of finite index normal subgroups of $F$ and $\mathfrak{N}$ be the corresponding relation-free $PQ$-category. Then the following are equivalent:
\begin{enumerate}
\item $\wh{F_\mathfrak{N}}$ is freely indexed
\item every $N\in \sfN$ admits an $L\in \sfN$ with $L\le N$ and $d(N)=d(N/L)$
\item every $G\in \mathfrak{N}$ admits $H\in\mathfrak{N}$ with $H\twoheadrightarrow G$ and
$$d(\ker(H\twoheadrightarrow G))=\vert G\vert(\vert A\vert-1)+1.$$
\end{enumerate}
\end{Prop}
\begin{proof} It is easy to check that conditions (2) and (3) are equivalent. For the equivalence of (1) and (2) one can use that in order to be freely indexed it suffices to check that the Schreier formula \ref{Schreier formula} holds for every open normal subgroup $\mU$ of $\wh{F_\mathfrak{N}}$ \cite[Lemma 2.5 (ii)]{LubotzkyDries}. So, let $\mU$ be open and normal and $U:=\mU\cap F$ where $F$ is the abstract subgroup of $\wh{F_\mathfrak{N}}$ generated by $A$; then $U\in \sfN$ and the formula \ref{Schreier formula} holds for $\mU$ if and only if $d(\mU)=d(U)$. But $\mU=\varprojlim U/N$ where the limit is taken over all $N\in \sfN$ for which $N\le U$. Since $d(U)=d(\varprojlim U/N)$ if and only if $d(U)=d(U/N)$ for some $N$ the claim follows.
\end{proof}
Finally, the existence of universal $S$-extensions provides a sufficient condition for the existence of dissolvers.
For $S$ a cyclic group of prime order the following is already contained in \cite{geometry0} (in a different language), the general case is \cite[Theorem 4.5]{geometry1}.
\begin{Prop}\label{universalSimpliesdissolver} For every $A$-generated group $G$ and every finite simple group $S$, the universal $S$-extension $G^{A,S}$ is a dissolver of $G$.
\end{Prop}

\subsection{Four crucial properties of a $PQ$-category ${\mathfrak{N}}$} In the preceding section, four possible properties of a $PQ$-category have turned out to be crucial.
\begin{Def}\label{crucial properties} A $PQ$-category $\mathfrak{N}$ is
\begin{enumerate}
\item \emph{locally extensible} if every $G\in\mathfrak{N}$ admits some finite simple group $S$ for which $G^{A,S}\in \mathfrak{N}$;
\item \emph{freely indexed} if every $G\in \mathfrak{N}$ admits some $H\in \mathfrak{N}$, with $H\twoheadrightarrow G$ such that
$$d(\ker(H\twoheadrightarrow G))=\vert G\vert(\vert A\vert-1)+1;$$
\item \emph{arboreous} if every $G \in\mathfrak{N}$ admits a (pre)dissolver $H\in \mathfrak{N}$;
\item \emph{Hall} if every $G \in\mathfrak{N}$ admits a weak (pre)dissolver $H\in \mathfrak{N}$.
\end{enumerate}
Moreover, a formation  $\mathfrak{F}$ or a variety $\mathbf{V}$ has the respective property if for every finite set $A$, $\vert A\vert\ge 2$, the $PQ$-category $\mathfrak{F}_A$ (resp.~$\mathbf{V}_A$) of all $A$-generated members of $\mathfrak{F}$ (of $\mathbf{V}$) has the property in question.
\end{Def}
The discussion so far has shown:
\begin{Thm} Let $\mathfrak{N}$ be a $PQ$-category.
\begin{enumerate}
\item $\wh{F_\mathfrak{N}}$ is freely indexed if and only if $\mathfrak{N}$ is freely indexed.
\item $\Gamma(\wh{F_\mathfrak{N}})$ is tree-like if and only if $\mathfrak{N}$ is arboreous.
\item  $\Gamma(\wh{F_\mathfrak{N}})$ is Hall if and only if $\mathfrak{N}$ is Hall.
\end{enumerate}
\end{Thm}

 By \cite{supersolvable,geometry1}, the following implications hold for a $PQ$-category:
$$\mbox{locally extensible} \Rightarrow \mbox{freely indexed} \Rightarrow \mbox{Hall}$$
and
$$\mbox{locally extensible} \Rightarrow \mbox{arboreous} \Rightarrow \mbox{Hall}.$$

In \cite{supersolvable} the varieties of finite supersolvable groups having either of the properties of Definition \ref{crucial properties} were classified and it was shown that the four conditions are equivalent in that case. Shusterman \cite{Shusterman} classified all prosupersolvable freely indexed groups and an email discussion with him revealed that even for $PQ$-categories of supersolvable groups these four conditions are equivalent (or equivalently, for finitely generated prosupersolvable groups the corresponding properties are the same). In particular, the formations of finite supersolvable groups enjoying either of these properties coincide with the varieties found in \cite{supersolvable}. To the best of our knowledge, no examples of $PQ$-categories or formations, let alone varieties, so far have been found for which these four conditions are distinct. The main purpose of the present paper is to construct examples of formations (via $PQ$-categories)  for which these four conditions disagree. Section \ref{non-solvable} presents examples of formations which are arboreous (and therefore Hall) but not freely indexed (and therefore not locally extensible), while Section \ref{solvable} is concerned with the construction of formations of solvable groups which are arboreous and freely indexed (and therefore Hall) but not locally extensible. The problem to construct such varieties and/or formations has been raised in several papers \cite{geometry1,geometry0,power,supersolvable} and, for varieties this is still open. The problem to classify all Hall formations has been proposed by Ballester-Bolinches, Pin and Soler-Escriv\`a \cite{ballpinxaro}.

\subsection{Pro-$\mathfrak{N}$-topology and geometry of $\Gamma(\wh{F_\mathfrak{N}})$} \subsubsection{Interplay between pro-$\mathfrak{N}$-topology of $F$ and geometry of $\Gamma(\wh{F_\mathfrak{N}})$} Let $\mathfrak{N}$ be a relation-free $PQ$-category.
Following Margolis, Sapir and Weil \cite{MSW} we call a finitely generated subgroup $H$ of $F$ \emph{$\mathfrak{N}$-extendible} if the core graph $\cH$ of $H$ can be embedded into a finite complete graph $\ol{\cH}$ whose transition group $T_{\ol{\cH}}$ belongs to $\mathfrak{N}$ (in order to be precise, this means: $T_{\ol{\cH}}$ is isomorphic --- \textbf{as an $A$-generated group} --- with a member of $\mathfrak{N}$). Since a permutation automaton $(\cA,1)$ (with distinguished vertex $1$) is always isomorphic to the Schreier graph $(\Sigma(T_\cA,U,A),U)$ (with distinguished vertex $U$) where $U$ is the stabilizer of $1$, the permutation automata $\cA$ for which $T_\cA$ belongs to $\mathfrak{N}$ are exactly the Schreier graphs $\Sigma(F,H,A)$ for finite index subgroups $H$ of $F$ for which there exists $N\in\sfN$ with $H\ge N$.

Let $\cH$ with distinguished vertex $1=1_\cH$ be the Stallings automaton of the finitely generated subgroup $H$ of $F$ and let $G\in\mathfrak{N}$. By ${\cH}^G$ we denote the subgraph of $\Gamma(G)$ spanned by all edges lying on some path starting at $1$ labeled by some word $w\in F$ which labels a path in $\cH$ starting at $1_\cH$. We set $\cH^{\wh{F_\mathfrak{N}}}:=\varprojlim_{G\in \mathfrak{N}}\cH^G$ and call $\cH^{\wh{F_\mathfrak{N}}}$ the \emph{$\mathfrak{N}$-universal covering graph of $\cH$} \cite{geometry1}. Suppose that $\cH$ embeds (as a pointed graph) in a complete graph $\cC$ (with distinguished vertex $1_\cC$) with transition group $T_\cC$; the unique graph morphism $\Gamma(T_\cC)\twoheadrightarrow \cC$ mapping $1$ to $1_\cC$ induces a unique morphism $\cH^{T_\cC}\twoheadrightarrow \cH$ mapping $1$ to $1_\cH$.  Hence, if $H$ is $\mathfrak{N}$-extendible then there exists a canonical morphism $\cH^{\wh{F_\mathfrak{N}}}\twoheadrightarrow \cH$ of pointed graphs. Suppose conversely that there is such a morphism  $\cH^{\wh{F_\mathfrak{N}}}\twoheadrightarrow \cH$; from the universal property of projective limits that morphism factors through $\cH^G$ for some $G\in \mathfrak{N}$, that is, there is a morphism $\tau:\cH^G\twoheadrightarrow \cH$ of pointed graphs mapping the distinguished vertex $1$ of $\cH^G$ to the distinguished vertex $1_\cH$ of $\cH$.  Let
$$T:=\{[w]_G\mid w\in H\}$$
be the image of $H$ under the canonical morphism $F\twoheadrightarrow G$. Consider the Schreier graph $\Sigma:=\Sigma(G,T,A)$ and the canonical morphism $\psi:\Gamma(G)\twoheadrightarrow \Sigma$. (Readers familiar with graph congruences and Stallings foldings \cite{MSW, kapmas, Stallings} may view the map $\psi$ as follows: first identify all vertices of $T$ to one vertex and then, in the resulting graph, apply Stallings foldings until the outcome is a folded graph.) In particular, $\psi(t)=T$ for all $t\in T$. Note that every such $t$ is a vertex of $\cH^G$. Now let $g$ and $h$ be vertices of $\cH^G$ such that $\psi(g)=\psi(h)$. Hence there are $s,t\in T$ and a word $w$ such that $g=s[w]_G$ and $h=t[w]_G$. Since $s$ and $g$ are vertices of $\cH^G$ there is a word $v$ labeling a path $\pi:s\to s[w]_G$ which runs entirely in $\cH^G$. In addition, $s[w]_G=g=s[v]_G$ which implies $[w]_G=[v]_G$. Under $\tau$, the path $\pi$ is mapped to the path $\tau(\pi):1_\cH\to \tau(g)$ having label $v$. From the definition of $\cH^G$ it follows that the path $\pi':t\to h=t[w]_G=t[v]_G$ labeled $v$ runs entirely in $\cH^G$ (since the path $\tau(\pi)$ may be lifted to a path in $\cH^G$ starting at $t$ instead of $s$). Then $\tau(\pi)=\tau(\pi')$ and, in particular, $\tau(g)=\tau(h)$. Altogether we have proved:
\begin{equation}\label{factoring tau}
\mbox{for all vertices }g,h\in \cH^G: \psi(g)=\psi(h)\Longrightarrow \tau(g)=\tau(h).
\end{equation}
(From the point of view of Stallings foldings, the above argument essentially says that for two vertices $g,h\in\cH^G$ which are identified  by the folding process (which eventually computes the map $\psi$), that process may be performed in a way such that for the identification of $g$ with $h$ only edges of $\cH^G$ are involved. This means one could start the folding process $\Gamma(G)\to \Sigma$ inside $\cH^G$ which leads to an intermediary graph which contains $\cH$ as subgraph; the subsequent foldings would leave this subgraph unchanged so that, at the end, $\cH$ appears as subgraph of $\Sigma$.) Implication  \ref{factoring tau} says that the map $\tau$ factors through $\psi(\cH^G)$, that is, there is a (unique) morphism of pointed graphs $\alpha:\psi(\cH^G)\twoheadrightarrow \cH$ such that $\alpha\circ \psi\vert_{\cH^G}=\tau$. On the other hand, every $w\in H$ labels a path $1\to [w]_G\in T$ running entirely in $\cH^G$. That path is mapped under $\psi$ to a path in $\psi(\cH^G)$ closed at $T$. Hence by \cite[Proposition 2.4]{MSW} there is a (unique) morphism of pointed graphs $\beta: \cH\to \psi(\cH^G)$ mapping $1_\cH$ to $T$. Uniqueness of the involved morphisms implies $\psi\vert_{\cH^G}=\beta\circ\tau$. Using the surjectivity of $\tau$ onto $\cH$ and the surjectivity of $\psi\vert_{\cH^G}$ onto $\psi(\cH^G)$, it follows that $\alpha\circ\beta$ resp.~$\beta\circ\alpha$ is the identity on $\cH$ resp.~on $\psi(\cH^G)$, so that $\alpha$ and $\beta$ are inverse isomorphisms between $\cH$ and $\psi(\cH^g)$. Hence we see that $\cH\cong \psi(\cH^G)$ appears as a pointed subgraph of $\Sigma$.   Since the transition group of $\Sigma$ is $G/T_G$ it belongs to $\mathfrak{N}$, that is, $\cH$ is $\mathfrak{N}$-extendible.  Altogether we get the following result \cite{geometry1}.
\begin{Thm} A finitely generated subgroup $H$ of $F$ is $\mathfrak{N}$-extendible if and only if there is a (unique) continuous morphism (of pointed graphs) $\cH^{\wh{F_\mathfrak{N}}}\twoheadrightarrow \cH$.
\end{Thm}
In view of this theorem one can give the following interpretation of the Stallings automaton $\til{\cH}$ of the $\mathfrak{N}$-extendible closure \cite{MSW} $\til{H}$ of a finitely generated subgroup $H$ of $F$ (having Stallings automaton $\cH$):  $\til{\cH}$ is  the largest quotient of $\cH$ which is also a (continuous) quotient of $\cH^{\wh{F_\mathfrak{N}}}$.

\begin{Rmk} The following Theorems \ref{3.12}, \ref{3.14}, \ref{RZTheorem}  are about the interplay between the pro-$\mathfrak{N}$-topology on $F$ and the geometry of the Cayley graph $\Gamma(\wh{F_\mathfrak{N}})$ of the pro-$\mathfrak{N}$-completion of $F$. These results have already been formulated and proved in \cite{geometry0, geometry1} in the context of formations \cite{geometry1} and varieties \cite{geometry0}. In that proofs only the geometric structure of the Cayley graph $\Gamma(\wh{F_\mathfrak{N}})$ and its approximation by finite graphs was used. The fact that the groups $\wh{F_\mathfrak{N}}$ were \emph{relatively free} played no role. Hence these proofs carry over verbally to the present more general context and are omitted. We also note that, as a prerequisite, the relevant statements in sections 2.1 and 2.2 of \cite{MSW} also carry over from the variety case to the present situation of $PQ$-categories. The reason for the latter is that they are based on Theorem 3.3 in \cite{Hall} (the pro-$\mathfrak{N}$-closure in $F$ of a finitely generated subgroup $H$ is the intersection of all pro-$\mathfrak{N}$-open subgroups containing $H$), which also holds in the present context.
\end{Rmk}

\begin{Thm}[\cite{geometry1}, Theorem 3.1]\label{3.12}  An $\mathfrak{N}$-extendible finitely generated subgroup $H$ of $F$ is pro-$\mathfrak{N}$-closed if and only if $\cH^{\wh{F_\mathfrak{N}}}$ is a Hall-subgraph of $\Gamma(\wh{F_\mathfrak{N}})$.
\end{Thm}

\begin{Lemma}\label{Hallclosure} Let $\cH$ be a connected subgraph of a connected profinite graph $\Gamma$; then there exists a (unique) smallest Hall-subgraph  $(\cH)^{\bm{\mathsf{H}}}$ containing $\cH$.
\end{Lemma}
\begin{proof} Let $\mathrm{Hall}(\cH)$ be the set of all Hall-subgraphs of $\Gamma$ which contain $\cH$ and let $\cK:=\bigcap \mathrm{Hall}(\cH)$. Then $\cK$ is Hall: let $u,v$ be vertices of $\cK$ which are connected by a finite reduced path $\pi$; for every  $\cC\in \mathrm{Hall}(\cH)$, since $\cK\subseteq \cC$, we have $u,v\in \cC$ and $\pi\subseteq \cC$, hence $\pi\subseteq \bigcap \mathrm{Hall}(\cH)=\cK$. Every connected component of a Hall-subgraph of $\Gamma$ is itself a Hall-subgraph. It follows that the connected component of $\cK$ containing $\cH$ is a connected Hall-subgraph of $\Gamma$ and is contained in every Hall-subgraph of $\Gamma$ containing $\cH$.
\end{proof}

\begin{Problem} Let $\mathfrak{N}$ be a relation-free $PQ$-category; for a subgroup $H$ of $F$ let us denote by $\ol{H}$ the closure in $F$ with respect to the pro-$\mathfrak{N}$ topology.
\begin{enumerate}
\item Is  $\ol{H}$ finitely generated whenever $H$ is finitely generated?
\item For an $\mathfrak{N}$-extendible subgroup $H$ of $F$ with core graph $\cH$, what is the connection between  $\ol H$ and $(\cH^{\wh{F_\mathfrak{N}}})^{\bm{\mathsf{H}}}$? In case $K=\overline H$ is finitely generated with core graph $\mathrsfs{K}$, is $\mathrsfs{K}^{\wh{F_\mathfrak{N}}}=(\mathrsfs{H}^{\wh{F_\mathfrak{N}}})^{\bm{\mathsf{H}}}$?
\end{enumerate}
\end{Problem}

Of particular interest are categories $\mathfrak{N}$ for which \textbf{every} $\mathfrak{N}$-extendible subgroup $H$ is pro-$\mathfrak{N}$-closed. This is the case if and only if \textbf{every} connected subgraph of $\Gamma(\wh{F_\mathfrak{N}})$ is a Hall-subgraph, that is $\Gamma(\wh{F_\mathfrak{N}})$ has the Hall property (Definition \ref{Hall-subgraph}).
\begin{Thm}[\cite{geometry1}, Theorem 3.2]\label{3.14}  Let $\mathfrak{N}$ be a relation-free $PQ$-category; then every $\mathfrak{N}$-extendible finitely generated subgroup $H$ of $F$ is pro-$\mathfrak{N}$-closed if and only if the Cayley graph $\Gamma(\wh{F_\mathfrak{N}})$ is Hall.
\end{Thm}

The property of being tree-like (Definition \ref{definition:tree-like}) has attracted considerable attention and is important  in the context of the Ribes--Zalesski\u\i-Theorem \cite{RZ}.

\begin{Thm}[\cite{geometry1}, Theorem 3.5, Theorem 3.9]\label{RZTheorem} The following conditions on a relation-free $PQ$-category $\mathfrak{N}$ are equivalent:
\begin{enumerate}
\item the product $H_1H_2$ of any two $\mathfrak{N}$-extendible subgroups $H_1,H_2$ of $F$ is pro-$\mathfrak{N}$ closed;
\item the product $H_1\cdots H_n$ of any finite number $n$ of $\mathfrak{N}$-extendible subgroups of $F$ is pro-$\mathfrak{N}$-closed;
\item the Cayley graph of $\Gamma(\wh{F_\mathfrak{N}})$ is tree-like.
\end{enumerate}
\end{Thm}
\subsubsection{Discussion and Problems} Theorem \ref{RZTheorem} provides \textbf{almost a characterization} of the profinite Hausdorff topologies on $F$ for which the Ribes--Zalesski\u\i-Theorem \cite{RZ} holds. Indeed, if $\Gamma(\wh{F_\mathfrak{N}})$ is tree-like then also Hall and therefore every $\mathfrak{N}$-extendible subgroup $H$ of $F$ is pro-$\mathfrak{N}$-closed; but by Margolis, Sapir and Weil \cite[Proposition 2.7]{MSW} every finitely generated pro-$\mathfrak{N}$-closed subgroup $H$ of $F$ is $\mathfrak{N}$-extendible (as already mentioned, the proof in \cite{MSW} holds for the more general context of $PQ$-categories). Hence, the properties of being $\mathfrak{N}$-extendible and of being pro-$\mathfrak{N}$-closed coincide; in particular, item (2) then implies that the product $H_1\cdots H_n$ of any finite number $n$ of finitely generated pro-$\mathfrak{N}$-closed subgroups of $F$ is closed. But the statement (2) of Theorem \ref{RZTheorem}  is formally stronger than the statement obtained by replacing ``$\mathfrak{N}$-extendible'' by ``$\mathfrak{N}$-closed''. In connection with statement (1) we are tempted to ask:
\begin{Problem} Can in statements (1) or (2) of Theorem \ref{RZTheorem} ``$\mathfrak{N}$-extendible'' be replaced by ``pro-$\mathfrak{N}$-closed''?
\end{Problem}
The most challenging open problem seems to be the following:
\begin{Problem}\label{arboreous=Hall} Is every Hall $PQ$-category arboreous? In other words, does the Hall property of the Cayley graph of a profinite group imply that it is tree-like?
\end{Problem}
The problem of whether there exist varieties which are Hall but not arboreous has been asked in several papers \cite{geometry0,power,supersolvable}, for formations this question was raised in \cite{geometry1}; recent experience by the authors indicate that even on the level of $PQ$-categories this question seems to be hard. Finally, the following problem seems to be also open:
\begin{Problem} Is every freely indexed $PQ$-category arboreous?
\end{Problem}

\section{Non-solvable formations}\label{non-solvable}
The main result in this section will be that every $A$-generated finite group $G$ admits a positive integer $N$ such that for every $n\ge N$ the alternating group $\mathds{A}_n$, subject to a properly chosen generating set $A$, is a predissolver of $G$. The approach in a sense combines an idea of Ash \cite{Ash} with Jordan's Theorem on primitive permutation groups. The result will be essentially based on our main combinatorial result:
\begin{Thm}\label{completing_to_alternating_group} Let $\cA$ be a connected incomplete inverse automaton on $m$ vertices and let  $q$ be the smallest prime larger than $m$. Then, for every $n\ge m+q+2$ there exists a permutation automaton $\cC_n$ on $n$ vertices  extending $\cA$ whose transition group $T_{\cC_n}$ is the alternating group $\mathds{A}_n$.
\end{Thm}

\begin{proof} We assume that $m\ge 3$; the cases $m\le 2$ can be checked individually and are irrelevant for the sequel. In particular, $q\ge 5$. Let $V$ be the set of vertices of $\cA$.  We choose an  integer  $k\ge 0$ and let
$$W:=\{x_1,\dots,x_q,y,z,t_1,\dots,t_k\}$$
be a set of $q+k+2$ new vertices (that is, $V\cap W=\emptyset$). Note that $\vert V\vert \le q-1$ and $q+2\le \vert W\vert$. We extend the graph $\cA$ to a graph $\cC_n$ on the vertex set $V\cup W$: choose a letter $a\in A$ which does not induce a total transformation on $V$ and a vertex $v\in V$ which is not the initial vertex of an edge labeled $a$. Then add the following edges:
\begin{itemize}
\item $v\overset{a}{\to} x_1$
\item $y\overset{a}{\to}x_2\overset{a}{\to} x_3 \overset{a}{\to}t_1 \overset{a}{\to} t_2 \overset{a}{\to}\cdots\overset{a}{\to}  t_{k-1}  \overset{a}{\to}t_k \overset{a}{\to}z \overset{a}{\to} y $
\item $x_1 \overset{b}{\to} x_2 \overset{b}{\to}\cdots \overset{b}{\to} x_q\overset{b}{\to}x_1$
\end{itemize}
for some $b\ne a$ (see Figure 1).
\begin{figure}[ht]
\begin{tikzpicture}
\filldraw  (0,1) circle (2pt);
\filldraw  (-1,0.5) circle (2pt);
\filldraw  (1,0.5) circle (2pt);
\filldraw  (2,1) circle (2pt);
\filldraw  (-2,1) circle (2pt);
\filldraw  (0,2.5) circle (2pt);
\filldraw  (1,3) circle (2pt);
\filldraw  (0,-1.5) circle (2pt);
\filldraw[gray] (-1,3) circle (2pt);
\node[above] at (-1,3) {$t_1$};
\node[below] at (0,-1.5) {$v$};
\node[below right] at (1,0.5) {$x_1$};
\node[below] at (-1,0.5) {$y$};
\node[below left] at (-2,1) {$z$};
\node[right] at (0,1) {$x_2$};
\node[below right] at (0,2.5) {$x_3$};
\node[below] at (1,3) {$x_4$};
\node[below right] at (2,1) {$x_q$};
\draw[dotted] (-3,-0.5)--(3,-0.5);
\node at (-3,1.75) {$W$};
\node at (-3,-1.5) {$V$};
\draw[->,thick] (-0.05,1.1) to [out=105, in=-105] (-0.05,2.4);
\draw[->,thick] (0.05,1.1) to [out=75, in=-75] (0.05,2.4);
\draw[->,thick] (1.9,0.95) --(1.1,0.55);
\draw[->,thick] (0.9,0.55)-- (0.1,0.95);
\draw[->,thick] (-0.9,0.55)--(-0.1,0.95);
\draw[->,thick] (-1.9,0.95)--(-1.1,0.55);
\draw[->,thick] (0.1,2.55)--(0.9,2.95);
\draw[->, thick, dotted] (-0.1,2.55)--(-0.9,2.95);
\draw[->,thick] (0.05,-1.4)--(0.95,0.4);
\draw[->,thick,dotted] (-1.1,3) .. controls (-3.3,3.2) and (-2.8,1.5).. (-2.05,1.1);
\draw[->,thick,dotted] (1.1,3) .. controls (3.3,3.2) and (2.8,1.5).. (2.05,1.1);
\node[below right] at (0.3,-0.5) {$a$};
\node[below] at (1.6,0.85) {$b$};
\node[below] at (0.5,0.85) {$b$};
\node[right] at (0.05,1.75) {$b$};
\node[above] at (0.5,2.75) {$b$};
\node at (2.5,2.4) {$b$};
\node[below] at (-1.6,0.85) {$a$};
\node[below] at (-0.5,0.85) {$a$};
\node[left] at (-0.05,1.75) {$a$};
\node[above] at (-0.5,2.75) {$a$};
\node at (-2.5,2.4) {$a$};
\end{tikzpicture}
\caption{ }
\end{figure}
Up to now, the action of $a^{\pm 1}$ is still undefined on $x_4,\dots,x_q$ and possibly on elements of $V$, that of $b^{\pm 1}$ is still undefined on $y,z,t_1,\dots,t_k$ and possibly elements of $V$; the actions of all other letters are undefined on all of $W$ and possibly elements of $V$. In particular, the actions of $a^{\pm 1}$ as well as of $b^{\pm 1}$ are undefined on at least two vertices. We can extend the actions of all letters to total even permutations on $V\cup W$ in a way that all $b$-cycles, except $x_1\to x_2\to\cdots\to x_q\to x_1$ have lengths (strictly) smaller than $q$. Thereby we get a permutation automaton $\cC_n$ which extends $\cA$. The transition group $T_{\cC_n}$ is a subgroup of the alternating group $\mathds{A}_n$ where $n=\vert V\vert +\vert W\vert= m+q+k+2$.

Since all $b$-cycles (except $x_1\to x_2\to\dots\to x_q\to x_1$) have lengths smaller than $q$, a certain power $b^r$ of $b$ consists entirely of the cycle $x_1\to\dots\to x_q\to x_1$ (and fixes all other vertices). Altogether, $T_{\cC_n}$ contains a cycle of prime length $q<\vert V\vert +\vert W\vert -2= m+q+k$. Since $\cC_n$ is obviously connected, $T_{\cC_n}$ acts transitively on $V\cup W$. We show that $T_{\cC_n}$ is primitive: it suffices to show that $T_{\cC_n}$ does not leave invariant any non-trivial equivalence relation $\mathrm P$. Suppose that $\mathrm P$ is an equivalence relation on $V\cup W$ left invariant under $T_{\cC_n}$.  Then all blocks of $\mathrm P$  have the same size. The cycle $x_1\to \dots x_q\to x_1$ of prime length $q$  either belongs to a single block of $\mathrm P$, or its elements belong to $q$ distinct blocks. Suppose that the latter is true and that $\mathrm{P}$ is not the identity relation.   Then $x_1\mathrel{\mathrm P}s$ for some $s\notin \{x_1,\dots, x_q\}$ (since no block can be a singleton). By construction, the $b$-cycle of $s$ is shorter than $q$, that is, there exists $l< q$ such that $s\cdot b^l=s$. Now $x_1\mathrel{\mathrm P} s$ implies
$$x_1\ne x_{l+1}=x_1\cdot b^l\mathrel{\mathrm P}s\cdot b^l=s\mathrel{\mathrm P} x_1,$$
that is, $x_1\ne x_{l+1}\mathrel{\mathrm{P}} x_1$, a contradiction to the assumption that all $x_i$ are in distinct blocks. So, suppose that all elements $x_1,\dots,x_q$ belong to a single block of $\mathrm P$. In particular, $x_2\mathrel{\mathrm P} x_3$ which implies $x_2\cdot a^l\mathrel{\mathrm P}x_3\cdot a^l$ for all $l\ge 1$. So we get:
\[
\begin{aligned}
x_2\mathrel{\mathrm P} x_3=x_2\cdot a \mathrel{\mathrm P}x_3\cdot a
    &= t_1=x_2\cdot a^2\mathrel{\mathrm P}x_3\cdot a^2\\
    &= t_2=x_2\cdot a^3\mathrel{\mathrm P} x_3\cdot a^3\\
    &\ \ \vdots\\
    &=t_k=x_2\cdot a^{k+1}\mathrel{\mathrm P}x_3\cdot a^{k+1}\\
    &=z=x_2\cdot a^{k+2}\mathrel{\mathrm P}x_3\cdot a^{k+2}=y.
\end{aligned}
\]
Altogether, all of $W$ would belong to a single block of $\mathrm P$;  since $\vert W\vert >\frac{\vert W\cup V\vert}{2}$ this implies that $\mathrm{P}$ has only one block. By Jordan's Theorem \cite[Theorem 3.3E]{Dixon and Mortimer}, $T_{\cC_n}=\mathds{A}_n$.

\end{proof}

We continue with some further prerequisites. For the following, we fix an $A$-generated group $G$ with Cayley graph $\Gamma=\Gamma(G)$.
The set of all constellations of $G$ is partially ordered by
$$(\Xi,g,\Theta)\le (\Xi',g',\Theta') \Longleftrightarrow \Xi\subseteq \Xi', g=g',\Theta\subseteq \Theta'.$$
A predissolver $H$ of some constellation is also a predissolver of every smaller constellation. It is therefore sufficient to consider maximal constellations. We give a description of the maximal constellations of $G$ in terms of cut sets of $\Gamma$. Consider a  \emph{minimal cut set} in $\Gamma$, that is, a set $C$ of (geometric) edges such that the graph $\Gamma\setminus C$ has two connected components, but $\Gamma\setminus C'$ is connected for every proper subset $C'$ of $C$ (recall that by a \emph{geometric edge} we mean a pair $f^{\pm 1}$ where $f$ is an edge of $\Gamma$). Now decompose the set $C$ into two disjoint  {non-empty} subsets: $C=C_\Xi\cup C_\Theta$ and set $$\Xi:=\Gamma\setminus C_\Theta\text{ and }\Theta:=\Gamma\setminus C_\Xi.$$
\begin{Prop} Let $\Xi,\Theta,C$, etc., be as above and let $g$ be a vertex not in the connected component of $1$ in $\Gamma\setminus C$; then $(\Xi,g,\Theta)$ is a maximal constellation, and conversely, every maximal constellation can be so constructed.
\end{Prop}
\begin{proof} It is clear that $(\Xi,g,\Theta)$ is maximal: the only possibility to extend it would be to add an edge of $C_\Xi$ to $\Theta$ or to add an edge of $C_\Theta$ to $\Xi$: in both cases the intersection of the resulting two graphs would be connected.

Let, conversely, $(\Xi,g,\Theta)$ be a maximal constellation; let $\Omega_1$ resp.~$\Omega_g$ be the connected component of $1$ resp.~$g$ in the graph $\Xi\cap\Theta$. By the maximality of $(\Xi,g,\Theta)$, $V(\Xi)=V(\Gamma)=V(\Theta)$ and $V(\Gamma)=V(\Omega_1)\cup V(\Omega_g)$. To see the former, assume w.l.o.g.~that $V(\Xi)\ne V(\Gamma)$. Then there exists a vertex $v$ and a geometric edge $e^{\pm 1}$ not in $\Xi$ such that $\Xi^*:=\Xi\cup \{e^{\pm 1},v\}$ is connected. By maximality,  {the graph $\Xi^*\cap \Theta$ has $1$ and $g$ in the same connected component}, hence $e^{\pm 1}$ connects $\Omega_1$ with $\Omega_g$ so that  {$v\in \Omega_1\cup\Omega_g\subseteq \Xi$}, a contradiction. To see the latter, note that if $\Xi\cap\Theta$ had another connected component $\Omega\notin\{\Omega_1,\Omega_g\}$, then $(\Xi\cup\{e^{\pm1}\},g,\Theta\cup\{e^{\pm1}\})$, where $e^{\pm1}$ is any geometric edge having one vertex in $\Omega$ and the other outside $\Omega$, would be a larger constellation in $\Gamma$,  {which again leads to a} contradiction. Now let $C_\Xi$ resp.~$C_\Theta$ be the set of all geometric edges of $\Xi$ resp.~$\Theta$ having one vertex in $\Omega_1$ and the other vertex in $\Omega_g$. Again by the maximality of $(\Xi,g,\Theta)$, we have $\Gamma=\Omega_1\cup\Omega_g\cup C_\Xi\cup C_\Theta$, so $C_\Xi\cup C_\Theta$ is a (clearly minimal) cut set in $\Gamma$ and $\Xi=\Gamma\setminus C_\Theta$ as well as $\Theta=\Gamma\setminus C_\Xi$ hold, as required.
\end{proof}

Let us now consider the set $\bm{\mathcal{MC}}$ of all pairs $(\Xi,\Theta)$ where $(\Xi,g,\Theta)$ is a maximal constellation for some $g$. For $(\Xi,\Theta)\in \bm{\mathcal{MC}}$  form the disjoint union $\Xi\sqcup \Theta$ and identify both basepoints $1$; the resulting graph is not folded, but we denote by $\Xi\mathrel{\underset{1}{\sqcup}}\Theta$ the largest folded quotient \cite{MSW, kapmas, Stallings}. This graph can be described differently: let $C=C_\Xi\cup C_\Theta$ be the minimal cut set of $\Gamma$ giving rise to the pair $(\Xi,\Theta)$ and let $\Delta\cup \Upsilon=\Gamma\setminus C$ where $\Delta$ is the connected component of $1$ and $\Upsilon$ is the other component (Figure 2).
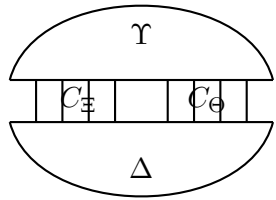
\begin{figure}[ht]
\begin{tikzpicture}[scale=0.7]
\draw[thick] (0,0)--(5,0);
\draw[thick] (0,0) to [out=-75,in=-105](5,0) ;
\draw[thick] (0.5,0)--(0.5,0.8);
\draw[thick] (1,0)--(1,0.8);
\draw[thick] (1.5,0)--(1.5,0.8);
\draw[thick] (2,0)--(2,0.8);
\draw[thick] (3,0)--(3,0.8);
\draw[thick] (3.5,0)--(3.5,0.8);
\draw[thick] (4,0)--(4,0.8);
\draw[thick] (4.5,0)--(4.5,0.8);
\draw[thick] (0,0.8)--(5,0.8);
\draw[thick] (0,0.8) to [out=75,in=105] (5,0.8);
\node[above] at (1.3,0) {$C_\Xi$};
\node[above] at (3.75,0) {$C_\Theta$};
\node[below] at (2.5,-0.5) {$\Delta$};
\node[above] at (2.5,1.3) {$\Upsilon$};
\end{tikzpicture}
\caption{$\Gamma=\Delta\cup C\cup \Upsilon$}
\end{figure}

Consider two disjoint copies of $\Upsilon$, denoted $\Upsilon_\Xi$ and $\Upsilon_\Theta$, and form the graph
$$\Upsilon_\Xi\cup C_\Xi\cup \Delta\cup C_\Theta\cup \Upsilon_\Theta$$ (Figure 3)
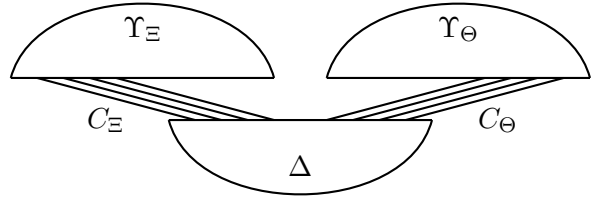
\begin{figure}[ht]
\begin{tikzpicture}[scale=0.7]
\draw[thick] (0,0)--(5,0);
\draw[thick] (0,0) to [out=-75,in=-105](5,0) ;
\draw[thick] (0.5,0)--(-2.5,0.8);
\draw[thick] (1,0)--(-2,0.8);
\draw[thick] (1.5,0)--(-1.5,0.8);
\draw[thick] (2,0)--(-1,0.8);

\draw[thick] (3,0)--(6,0.8);
\draw[thick] (3.5,0)--(6.5,0.8);
\draw[thick] (4,0)--(7,0.8);
\draw[thick] (4.5,0)--(7.5,0.8);
\draw[thick] (-3,0.8)--(2,0.8);
\draw[thick] (-3,0.8) to [out=75,in=105] (2,0.8);
\node at (-1.2,0) {$C_\Xi$};
\node at (6.25,0) {$C_\Theta$};
\node[below] at (2.5,-0.5) {$\Delta$};
\node[above] at (-0.5,1.3) {$\Upsilon_\Xi$};
\node[above] at (5.5,1.3) {$\Upsilon_\Theta$};
\draw[thick] (3,0.8)--(8,0.8);
\draw[thick] (3,0.8)to [out=75,in=105] (8,0.8);
\end{tikzpicture}
\caption{$\Xi\mathrel{\protect\underset{1}{\sqcup}}\Theta=\Upsilon_\Xi\cup C_\Xi\cup\Delta\cup C_\Theta\cup \Upsilon_\Theta$}
\end{figure}
where the edges $C_\Xi$ connect appropriate vertices of $\Delta$ and $\Upsilon_\Xi$ while those of $C_\Theta$ connect vertices of $\Delta$ with those of $\Upsilon_\Theta$  such that
$$ \Xi\cong\Delta\cup C_\Xi\cup \Upsilon_\Xi\text{ and }\Theta\cong\Delta\cup C_\Theta\cup \Upsilon_\Theta.$$
(Recall that $\Xi=\Delta\cup C_\Xi\cup \Upsilon$ and $\Theta=\Delta\cup C_\Theta\cup \Upsilon$.)

 It is easy to see that $\Xi\mathrel{\underset{1}{\sqcup}}\Theta$ is incomplete (at least one letter $a$ induces a non-total transformation); for example, the copy in $\Upsilon_\Theta$ of the end-vertex in $\Upsilon$ of every edge of $C_\Xi$ admits a letter $a\in A$ for which the induced transformation or its inverse is not defined. From the construction, the following is immediate.
\begin{Prop} \label{witness}\begin{enumerate}
\item The transition group $T$ of any completion of $\Xi\mathrel{\underset{1}{\sqcup}}\Theta$ is a predissolver of every maximal constellation of the form $(\Xi,g,\Theta)$ (and hence of every smaller constellation).
\item If $\cA$ is an inverse automaton which contains $\Xi\mathrel{\underset{1}{\sqcup}}\Theta$ for every $(\Xi,\Theta)\in \bm{\mathcal{MC}}$ as a subgraph then the transition group $T$ of any completion of $\cA$ is a predissolver of every constellation of $G$.
\end{enumerate}
\end{Prop}
\begin{proof} For (1), let $(\Xi,g,\Theta)$ be a maximal constellation and $C_\Xi, C_\Theta, \Delta, \Upsilon$ be such that  $\Gamma=\Delta\cup (C_\Xi\cup C_\Theta)\cup \Upsilon$, $\Xi=\Delta\cup C_\Xi\cup \Upsilon$ and $\Theta=\Delta\cup C_\Theta\cup \Upsilon$ as in the construction. Then $g\in \Upsilon$; denote the $\Xi$- and $\Theta$-version, respectively, of $g$ in  $\Xi\mathrel{\underset{1}{\sqcup}}\Theta$, that is, in $\Upsilon_\Xi$ respectively $\Upsilon_\Theta$, by $g_\Xi$ and $g_\Theta$. Let $u, v\in F$ with $[u]_G=g=[v]_G$ be such that the path $u:1\to g$ in $\Gamma$ runs inside $\Xi$ while the path $v:1\to g$ in $\Gamma$ runs inside $\Theta$. Then, in $\Xi\mathrel{\underset{1}{\sqcup}}\Theta$, $1\cdot u=g_\Xi\ne g_\Theta=1\cdot v$. It follows that in the transition group  $T$ of any completion of $\Xi\mathrel{\underset{1}{\sqcup}}\Theta$, $[u]_T\ne[v]_T$. The idea to this construction and argument comes from Ash's paper \cite{Ash}.

(2) As in (1), if $u,v$ are words as in (1) for some maximal constellation $(\Xi,g,\Theta)$ of $G$, then since $\Xi\mathrel{\underset{1}{\sqcup}}\Theta$ is a subgraph of $\cA$, there is a vertex $h$ in $\cA$ for which $h\cdot u\ne h\cdot v$. For the same reason as in (1), $[u]_T\ne[v]_T$ for the transition group $T$ of every completion of $\cA$.
\end{proof}
In order to continue, denote the set $\{\Xi\mathrel{\underset{1}{\sqcup}}\Theta\mid (\Xi,\Theta)\in \bm{\mathcal{MC}}\}$ by $\bm{\mathcal{AMC}}$. Now, for every $\Xi\mathrel{\underset{1}{\sqcup}}\Theta\in \bm{\mathcal{AMC}}$ choose a letter $a\in A$ which induces a non-total transformation on $\Xi\mathrel{\underset{1}{\sqcup}}\Theta$. This induces a partition of the set $\bm{\mathcal{AMC}}$: for each $a\in A$ let $\bm{\mathcal{AMC}}_a$ be the set of all members $\Xi\mathrel{\underset{1}{\sqcup}}\Theta$ of $\bm{\mathcal{AMC}}$ for which we have chosen the letter $a$. In case $\bm{\mathcal{AMC}}_a$ is empty it should be ignored in the following arguments. (Under certain conditions, e.g. if $A$ is irredundant,  for every $a\in A$ there exists some $\Xi\mathrel{\underset{1}{\sqcup}}\Theta\in \bm{\mathcal{AMC}}$ for which the only choice is $a$). Now denumerate the members of $\bm{\mathcal{AMC}}_a$ somehow, say,
$$\Xi_{a1}\mathrel{\underset{1}{\sqcup}}\Theta_{a1}, \Xi_{a2}\mathrel{\underset{1}{\sqcup}}\Theta_{a2},\dots,\Xi_{at_a}\mathrel{\underset{1}{\sqcup}}\Theta_{at_a}.$$
Next form the disjoint union of these graphs and produce a connected folded graph by adding appropriate edges:
$$\Xi_{a1}\mathrel{\underset{1}{\sqcup}}\Theta_{a1}\mathrel{\overset{a}{\to} } \Xi_{a2}\mathrel{\underset{1}{\sqcup}}\Theta_{a2}\mathrel{\overset{a}{\to}}\dots\mathrel{\overset{a}{\to}}\Xi_{at_a}
\mathrel{\underset{1}{\sqcup}}\Theta_{at_a}$$
(recall that for every $j$, $a$ induces a non-total transformation on the graph $\Xi_{aj}\mathrel{\underset{1}{\sqcup}}\Theta_{aj}$) and denote the resulting graph by $\mathrsfs{AMC}_a$ ($a$ is still a non-total transformation on $\mathrsfs{AMC}_a$).

Finally, form the disjoint union of all $\mathrsfs{AMC}_a$, $a\in A$, add a new vertex $s$ and for all $a$ add an appropriate edge $\mathrsfs{AMC}_a\mathrel{\overset{a}{\to}}s$  and denote the resulting graph $\cA G$, which is a connected incomplete inverse automaton. From the construction and Proposition \ref{witness}, the following is immediate.
\begin{Lemma}\label{total_predissolver}\begin{enumerate}
\item The connected inverse automaton $\cA G$ contains every $\Xi\mathrel{\underset{1}{\sqcup}}\Theta\in \bm{\mathcal{AMC}}$ as a subgraph.
\item The transition group of any completion $\cC$ of $\cA G$ is a predissolver of every constellation $(\Xi,g,\Theta)$ of $G$, hence a predissolver of $G$.
\end{enumerate}
\end{Lemma}
Combination of Theorem \ref{completing_to_alternating_group} and Lemma \ref{total_predissolver} leads to:
\begin{Thm}\label{alternating-group-is-predissolver}
\begin{enumerate}
\item For every finite $A$-generated group $G$ there exists a positive integer $N\ge|G|$ such that for every $m\ge N$ the $A$-generated alternating group $\mathds{A}_m$ (subject to appropriately chosen generators) is a predissolver of $G$.
\item In case of (1), the product $\mathds{A}_m\mathrel{{\underset{A}{\times}}}G$ is a dissolver of $G$ and the kernel of the projection $\mathds{A}_m\mathrel{{\underset{A}{\times}}}G\twoheadrightarrow G$ is $\mathds{A}_m$ (which is a $2$-generated group).
\end{enumerate}
\end{Thm}
\begin{proof} Assertion (1) is an immediate consequence of Theorem \ref{completing_to_alternating_group} and Lemma \ref{total_predissolver}. For (2) we note that
$$\mathds{A}_m\mathrel{{\underset{A}{\times}}}G=\{([w]_{\mathds{A}_m},[w]_G)\mid w\in F\}.$$
Since $\vert G\vert\le N\le m$, in particular $\vert G\vert <\vert \mathds{A}_m\vert$. It follows that $G$ cannot project onto $\mathds{A}_m$, hence there exists a word $w\in F$ such that $[w]_{\mathds{A}_m}\ne1$ and $[w]_G=1$. Let
$$K:=\{x\in \mathds{A}_m\mid (x,1)\in \mathds{A}_m\mathrel{{\underset{A}{\times}}}G\}.$$
It is readily checked that $K \unlhd \mathds{A}_m$ whence $K=\mathds{A}_m$ since $K\ne\{1\}$; altogether $\mathrm{ker}(\mathds{A}_m\mathrel{{\underset{A}{\times}}}G\twoheadrightarrow G)\cong K=\mathds{A}_m$.
\end{proof}

We present some applications. Consider the profinite group
 $$\mA:=\prod_{k\ge 5} \mathds{A}_k.$$
 \begin{Cor} For every $n\ge 2$ the profinite group $\mA$ admits a generating set $A$ of size $n$ with respect to which the Cayley graph $\Gamma(\mA)$ is tree-like.
 \end{Cor}
 \begin{proof}
 We need to find a generating set $A$ of size $n$ such that  every ($A$-generated) finite quotient $G$ admits an ($A$-generated) quotient $H$ which dissolves all constellations of $G$ (with respect to $A$). Let us start with some arbitrary choice $A$ of generators of $\mathds{A}_5$, that is, we consider $\mathds{A}_5$ as an $A$-generated group for some generating set of size $n$ (note that $n$ may be bigger than $60=\vert\mathds{A}_5\vert$, that is, we actually choose a mapping $\vp_5:A\to\mathds{A}_5$ such that $\vp_5(A)$ generates $\mathds{A}_5$). By Theorem \ref{alternating-group-is-predissolver} there exists $n_1>5$ such that subject to a suitable choice of generators $\vp_{n_1}:A\to\mathds{A}_{n_1}$ the group $\mathds{A}_{n_1}$ is a predissolver of $\mathds{A}_5$  {(to be precise: a predissolver of $(\mathds{A}_5,\vp_5)$)}. For every $k\ge 5$ set $B_k:=\prod_{i=5}^k \mathds{A}_i$. For $i= 6,\dots, n_1-1$ take, in every $\mathds{A}_i$, an arbitrary generating set $A$ of size $n$ via a suitable mapping $\vp_i:A\to \mathds{A}_i$; we form the direct product over all $i$ between $5$ and $n_1$ to get the group $B_{n_1}$  and the mapping $\vp_5\times\cdots\times \vp_{n_1}:A\to B_{n_1}$. As such, $B_{n_1}$ is an $A$-generated group. This can be seen, for example, by induction and the use of Proposition 2.7 in \cite{Collins} which (in the terminology of the present paper) states that the product $G\mathrel{\underset{A}{\times}}H$ of two $A$-generated groups which have no common non-trivial isomorphic quotients coincides with the full Cartesian product $G\times H$. By construction, $B_{n_1}$ is a dissolver of $\mathds{A}_5$. Next, there exists $n_2>n_1$ such that subject to a suitable choice of a generating set $A$, $\mathds{A}_{n_2}$ is a predissolver of $B_{n_1}$; again take arbitrary generating sets $A$ of size $n$ in $\mathds{A}_i$ for $i=n_1+1,\dots, n_2-1$ and form the $A$-generated group $B_{n_2}$, which then is a dissolver of $B_{n_1}$. We may continue this process by induction and obtain a generating set $A$ of $\mA$ such that $\Gamma(\mA)$ is tree-like.
 \end{proof}
 In particular, this implies that for every finite alphabet $A$ which contains at least two letters, the free group $F$ on $A$ has normal subgroups $N_i$ for $i\ge 5$ such that $F/N_i\cong \mathds{A}_i$ for every $i$, and such that, setting $\sfN$ the set of all intersections of finitely many of the groups $N_i$ then the Cayley graph of $\wh{F_\mathfrak{N}}$ is tree-like and the Ribes--Zalesski\u\i-Theorem holds for the  pro-$\mathfrak{N}$ topology of $F$ (and $\mathfrak{N}=\{F/N\mid N\in \sfN\}$).

 This also shows: in an $A$-generated profinite group $\wh{F_\mathfrak{N}}$ for which $\Gamma(\wh{F_\mathfrak{N}})$ is tree-like it may happen that $d(\wh{F_\mathfrak{N}})<\vert A\vert$. On the other hand, a generating set of smallest size in such a group not necessarily gives rise to a tree-like Cayley graph. Indeed, in every alternating group $\mathds{A}_n$ the permutation $(123)$ can be extended to a generating pair of $\mathds{A}_n$; it follows that $\mA$ admits a generating pair $a,b$ for which $a^3=1$.

 Immediately from Theorem \ref{alternating-group-is-predissolver} we get:
 \begin{Cor} Every formation which contains the alternating group $\mathds{A}_n$ for infinitely many $n$ is arboreous and therefore Hall.
 \end{Cor}
 Let $\fF_e$ be the formation generated by $\mathds{A}_{2n}$ with $n\ge 3$ and $\fF_o$ be the formation generated by $\mathds{A}_{2n+1}$ with $n\ge 2$. Both formations are arboreous, but also their join $\fF:=\fF_e\vee \fF_o$ is arboreous. This is in marked contrast to varieties: an arboreous variety must be join irreducible \cite{geometry0,power}. These formations are not freely indexed and therefore not locally extensible. Moreover, $\mathfrak{F}_o\cap\mathfrak{F}_e=\{\mathbf 1\}$, the trivial formation. For the $A$-generated free profinite objects we have $$\wh{F_{\fF}}=\wh{F_{\fF_e}}\mathrel{\underset{A}{\times}} \wh{F_{\fF_o}}=\wh{F_{\fF_e}}\times \wh{F_{\fF_o}},$$ a full direct product of two $A$-generated profinite groups with tree-like Cayley graphs which has again a tree-like Cayley graph. This is a bit surprising given the result \cite[Theorem 3.16]{supersolvable}, which implies that the direct product of two $A$-generated prosolvable groups (with $\vert A\vert \ge 2$) having no common finite quotients can never have a tree-like Cayley graph.

 \section{Solvable formations}\label{solvable} For a given alphabet $A$ we first construct a relation-free $PQ$-category of solvable groups which is freely-indexed and arboreous (therefore Hall) but not locally extensible. This gives rise to a formation having the same properties. Every inverse sequence of finite groups $\cdots\twoheadrightarrow G_2\twoheadrightarrow G_1 \twoheadrightarrow G_0$ gives rise to a $PQ$-category (by taking all quotients of all $G_n$); the $PQ$-categories of the present section will be expressed in terms of such \emph{generating sequences}.

 The approach will be based on the Gasch\"utz extension. Given an $A$-generated group $G$ and a prime $p$ then by Proposition \ref{universalSimpliesdissolver}, $G^{\mathbb{Z}/p}$ is a dissolver of $G$. Hence, if we iterate this extension then we get a sequence which generates a locally extensible (and thus arboreous, in particular relation-free) $PQ$-category. We shall modify this construction as follows: we replace $G^{\mathbb{Z}/p}$ with the quotient by its center, that is, we consider the ``less powerful'' group $G^{\mathbb{Z}/p}/Z(G^{\mathbb{Z}/p})$. The latter group can be seen to be not even a weak dissolver of $G$. However it is powerful enough that three iterations yield a dissolver. Hence this construction produces a sequence $(G_n)$ which generates an arboreous $PQ$-category. If we choose distinct primes for the iteration then we will be able to show that the $PQ$-category generated this way is in addition freely indexed but not locally extensible.

\subsection{$PQ$-categories} \subsubsection{The center of the Gasch\"utz extension} For a prime $p$, recall the definition of the {universal Gasch\"utz $p$-extension} (Section \ref{universal S extension}). From this it follows that the extension $G\mapsto G^{\mathbb{Z}/p}$ is \emph{functorial} in the following sense.
\begin{Prop} \label{Gaschuetz functorial} Let $G\twoheadrightarrow H$ be $A$-generated groups; then the following diagram commutes
\[
\begin{tikzpicture}
\matrix (m) [matrix of math nodes, row sep=3em,
column sep=3em]
{ G^{\mathbb{Z}/p} & H^{\mathbb{Z}/p} \\
G & H \\
};
\path[->>]
(m-1-1) edge (m-1-2)
(m-1-1) edge (m-2-1)
(m-1-2) edge (m-2-2)
(m-2-1) edge (m-2-2);
\end{tikzpicture}
\]
where all arrows denote canonical morphisms.
\end{Prop}
The set of positive edges of the Cayley graph $\Gamma(G)$ of $G$ is denoted $E$ and will be identified with $G\times A$. We first use the representation of $G^{\mathbb{Z}/p}$ as a subgroup of $\mathbb{F}_p[E]\rtimes G$. For an element $\alpha\in \mathbb{F}_p[E]$ we denote by $\alpha(e)$ the coefficient of $e$ in $\alpha$, that is, $\alpha=\sum_{e\in E}\alpha(e)e$.
\begin{Prop}\label{center of Gaschuetz} An element $(\alpha,g)\in G^{\mathbb{Z}/p}$ belongs to the center  if and only if
\begin{enumerate}
\item $g=1$
\item for all $a\in A$ and all $h,k\in G: \alpha(h,a)=\alpha(k,a)$.
\end{enumerate}
\end{Prop}
\begin{proof}
Condition (2) says that $\alpha$ is constant on all edges having the same label. Assume that $(\alpha,g)$ satisfies (1) and (2) and let $(\beta,h)\in G^{\mathbb{Z}/p}$. Assumption (2) implies that ${}^h\alpha=\alpha$ whence
$$(\alpha,1)(\beta,h)=(\alpha+\beta,h)=(\beta+{}^h\alpha,h)=(\beta,h)(\alpha,1).$$
Let conversely $(\alpha,g)\in Z(G^{\mathbb{Z}/p})$. We first show by contradiction that $g=1$, so assume $g\not=1$. For each $a\in A$, we have
$$(\alpha+(g,a),ga)=(\alpha,g)\cdot((1,a),a)=((1,a),a)\cdot(\alpha,g)=((1,a)+{}^a\alpha,ag),$$
whence
\begin{equation}\label{center}
\alpha-{}^a\alpha=(1,a)-(g,a)\end{equation}
follows. Write
$$\alpha=\sum_{h\in G,b\in A}{\alpha(h,b)(h,b)},$$
so that
$${}^a\alpha=\sum_{h\in G,b\in A}{\alpha(h,b)(ah,b)}=\sum_{h\in G,b\in A}{\alpha(a^{-1}h,b)(h,b)}.$$
From  this, (\ref{center}) implies the following three facts:

\begin{enumerate}
\item $\alpha(a^{-1}h,b)=\alpha(h,b)$ for all $h\in G$ and $b\in A\setminus\{a\}$.
\item $\alpha(a^{-1}h,a)=\alpha(h,a)$ for all $h\in G\setminus\{1,g\}$.
\item $\alpha(1,a)-\alpha(a^{-1},a)=1$ and $\alpha(g,a)-\alpha(a^{-1}g,a)=-1$.
\end{enumerate}

We now argue that this implies $g\in\langle a\rangle$. Indeed, otherwise, by the second and third fact above, we get that
$$\alpha(1,a)=\alpha(a^{-1},a)+1=\alpha(a^{-2},a)+1=\cdots=\alpha(a^{-\ord(a)},a)+1=\alpha(1,a)+1,$$
a contradiction. So we can write $g=a^{-k}$ with $k\in\{1,\ldots,\ord(a)-1\}$, and the second and third facts yield
$$\alpha(1,a)=\alpha(a^{-1},a)+1=\cdots=\alpha(a^{-k},a)+1=\alpha(g,a)+1.$$
But likewise, for any $b\in A\setminus\{a\}$, we have $g\in\langle b\rangle$, so we can write $g=b^{-l}$ with $l\in\{1,\ldots,\ord(b)-1\}$, and iterated application of the first fact above, with swapped roles of $a$ and $b$, yields
$$\alpha(1,a)=\alpha(b^{-1},a)=\cdots=\alpha(b^{-l},a)=\alpha(g,a),$$
a contradiction. This concludes the proof that $g=1$. Now suppose that for some $h,k\in G$, $a\in A$, $\alpha(h,a)\not=\alpha(k,a)$; choose any $\beta\in\mathbb{F}_p[E]$ such that $(\beta,kh^{-1})\in G^{\mathbb{Z}/p}$. Then
$$(\alpha,1)(\beta,kh^{-1})=(\alpha+\beta,kh^{-1})$$
and
$$(\beta,kh^{-1})(\alpha,1)=(\beta+{}^{kh^{-1}}\alpha,kh^{-1}).$$
Now,
$$(\alpha+\beta)(h,a)=\alpha(h,a)+\beta(h,a)$$
while
$$(\beta+{}^{hk^{-1}}\alpha)(h,a)=\beta(h,a)+{}^{kh^{-1}}\alpha(h,a)=\beta(h,a)+\alpha(k,a).$$
Hence
$$\alpha+\beta\not=\beta+{}^{kh^{-1}}\alpha$$
whence
$$(\alpha,1)(\beta,kh^{-1})\not=(\beta,kh^{-1})(\alpha,1).$$
\end{proof}
\begin{Cor} For every choice of a function $f:A\to \mathbb{F}_p$ there exists $(\alpha,1)\in Z(G^{\mathbb{Z}/p})$ such that for all $a\in A$ and $g\in G$, $\alpha(g,a)=f(a)$. In particular, the center $Z(G^{\mathbb{Z}/p})$ is a rank-$\vert A\vert$ subgroup  of $\ker(G^{\mathbb{Z}/p}\twoheadrightarrow G)$ .
\end{Cor}
\begin{proof} Fix $a\in A$ and consider the decomposition of $G$ into its $a$-cycles (this is actually the subgraph of $\Gamma(G)$ spanned by all edges labeled $a$). On each cycle choose a vertex, thereby getting vertices $v_1,\dots, v_k$ (where $k$ is the number of $a$-cycles). Now, for every $i$ choose a word $w_i$ labeling a path in $\Gamma(G)$ from $1$ to $v_i$; suppose that the order of $a$ in $G$ is $e$. Consider the word
$$w_a:=w_1a^{ef(a)}w_1^{-1}\cdots w_ka^{ef(a)}w_k^{-1}$$
(where $f(a)$ is interpreted in $\{0,\dots,p-1\}$). The path labeled by $w_a$ starting at $1$ consecutively runs to every $a$-cycle, traverses this cycle $f(a)$ times and then runs back to $1$ via the same path. Evaluation of $w_a$ in $G^{\mathbb{Z}/p}$ gives
$$[w_a]_{G^{\mathbb{Z}/p}}=(\alpha,1)$$
where
$$\alpha(g,b)=\begin{cases} f(a)\mbox{ if }a=b\\ 0 \mbox{ else.}\end{cases}$$
Now produce, for every $a\in A$ a word $w_a$ of that kind and set $w:=\prod_{a\in A}w_a$ (the order of the factors is not relevant). We obtain
$$[w]_{G^{\mathbb{Z}/p}}=(\alpha,1)$$ where $\alpha(g,a)=f(a)$ for all $g\in G, a\in A$, as required.
\end{proof}
\begin{Notation} We set $G^{\widetilde{\mathbb{Z}/p}}:=G^{{\mathbb{Z}/p}}/Z(G^{{\mathbb{Z}/p}})$.
\end{Notation}
Next we formulate a criterion when a word evaluates as $1$ in either of the groups $G^{{\mathbb{Z}/p}}$ and  $G^{\widetilde{\mathbb{Z}/p}}$. The result is immediate from the representation of $G^{{\mathbb{Z}/p}}$ as subgroup of $\mathbb{F}_p[G\times A]\rtimes G$ and from Proposition \ref{center of Gaschuetz}. For a path $\pi$ in a graph $\Gamma$ and for an edge $e$ of $\Gamma$ we denote by $\pi(e)$ the number of signed traversals of $e$ by $\pi$; for a word $w\in F$, a group $G$ and $g\in G$ we denote by $\pi_g^G(w)$ the unique path in $\Gamma(G)$ starting at $g$ and being labelled by $w$.

\begin{Cor} \label{word problem} For $w\in F$,
\begin{enumerate}
\item $[w]_{G^{\mathbb{Z}/p}}=1 \Longleftrightarrow [w]_G=1 \text{ and for all } (g,a)\in G\times A$: $$\pi_1^G(w)(g,a)\equiv 0\ (\bmod\  p)$$
\item $[w]_{G^{\widetilde{\mathbb{Z}/p}}}=1\Longleftrightarrow [w]_G=1\text{ and for all } g,h\in G, a\in A$: $$\pi_1^G(w)(g,a)\equiv \pi_1^G(w)(h,a)\ (\bmod\ p).$$
\end{enumerate}
\end{Cor}
\subsubsection{Dissolving constellations}
\begin{Lemma} \label{key lemma} Let $G$ be a group, $\{1\}\ne K\le G$ be a non-trivial subgroup and $p$ be a prime; let $L\le G^{\widetilde{\mathbb{Z}/p}}$ be the inverse image of $K$ under the canonical map $G^{\widetilde{\mathbb{Z}/p}}\twoheadrightarrow G$. Let $\Gamma$ be the Cayley graph of $G^{\widetilde{\mathbb{Z}/p}}$, $g\in G^{\widetilde{\mathbb{Z}/p}}$ and $e=(g,a)$. Then the graph $\Gamma \setminus Le^{\pm 1}$ is disconnected; more precisely: $g$ and $ga$ are in distinct connected components.
\end{Lemma}
\begin{proof} A word $w\in F$ labels a path $g\to ga$ in $\Gamma\setminus Le^{\pm 1}$ if and only if $w$ labels a path $1\to a$ in $g^{-1}(\Gamma\setminus Le^{\pm 1})=\Gamma\setminus g^{-1}Lg(1,a)^{\pm 1}$. Since $g^{-1}Lg=\vp^{-1}(\bar g^{-1}K\bar g)$ for $\bar{g}=\vp(g)$ it suffices to treat the case $e=(1,a)$.

Suppose, by contradiction, that there is a word $w\in F$ which labels a path $1\to a$ in $\Gamma\setminus Le^{\pm 1}$; then $wa^{-1}$ labels a closed path in $\Gamma$, hence $[wa^{-1}]_{G^{\widetilde{\mathbb{Z}/p}}}=1$. On the other hand, the path $w:1\to a$ in $\Gamma\setminus Le^{\pm 1}$ projects to a path $w:1\to a$ in the graph $\Gamma(G)\setminus Ke^{\pm 1}$. The closed path $1\to 1$ in $\Gamma(G)$ labelled $wa^{-1}$ does not traverse any edge of the form $(k,a)$ for $1\ne k\in K$ (in either direction) and it traverses $(1,a)$ exactly once (in the reverse direction). Altogether, in the Cayley graph $\Gamma(G)$ we have:
$$\forall\ k\in K,k\ne 1: \pi_1^G(wa^{-1})(k,a)=0\not\equiv -1 = \pi_1^G(wa^{-1})(1,a)\ (\bmod\ p).$$
From Corollary \ref{word problem} (2) it follows that $[wa^{-1}]_{G^{\widetilde{\mathbb{Z}/p}}}\ne 1$, a contradiction.
\end{proof}
The following statement is obvious.
\begin{Lemma} \label{counting lifts} Let $\vp:H \twoheadrightarrow G$, $e\in \Gamma(G)$ and $w\in F$; then
$$\pi_1^G(w)(e)=\sum_{f\in \vp^{- 1}(e)}\pi_1^H(w)(f).$$
\end{Lemma}
Let $(\Xi,g,\Theta)$ be a constellation in the group $G$; let $\Upsilon$ be the connected component containing 1 of the graph $\Xi\cap \Theta$. Set:
$$\partial^+_\Xi:=\{e\in \Xi\cap E(\Gamma(G))\mid \iota(e)\in \Upsilon, \tau(e)\notin \Upsilon\}$$
and
$$\partial^-_\Xi:=\{e\in \Xi\cap E(\Gamma(G))\mid \iota(e)\notin \Upsilon, \tau(e)\in \Upsilon\}$$
and define $\partial^+_\Theta$ and $\partial^-_\Theta$ analogously. We note that
$$(\partial^+_\Xi\cup\partial^-_\Xi)\cap \Theta=\emptyset = (\partial^+_\Theta\cup\partial^-_\Theta)\cap \Xi.$$ For every word $w\in F$ for which $[w]_G=g$ and $\pi_1^G(w)\subseteq \Xi$, the path $\pi_1^G(w)$ traverses the ``border'' $\partial^+_\Xi\cup\partial^-_\Xi$ exactly one time more often in the forward direction than in the backward direction. Therefore (as in \cite[proof of Theorem 2.1]{mytype2} or \cite[p 160, (4.2)]{geometry1}):
$$\sum_{e\in \partial^+_\Xi}\pi_1^G(w)(e)-\sum_{f\in \partial^-_\Xi}\pi_1^G(w)(f)=1.$$
Immediately from Lemma \ref{counting lifts} we get
\begin{Cor}\label{detecting edges} Let $\vp:H\twoheadrightarrow G$ and $(\Xi,g,\Theta)$ be a constellation in $G$; then for every word $w\in F$ for which $[w]_G=g$ and $\pi_1^G(w)\subseteq \Xi$ we have
$$\sum_{e\in \vp^{-1}(\partial^+_\Xi)}\pi_1^H(w)(e)-\sum_{f\in \vp^{-1}(\partial^-_\Xi)}\pi_1^H(w)(f)=1.$$
\end{Cor}
The main result now is:
\begin{Thm} Let $p,q,r$ be (not necessarily distinct) primes, $G$ be a group and $H=(G^{\widetilde{\mathbb{Z}/p}})^{\widetilde{\mathbb{Z}/q}}$. Then $H^{\widetilde{\mathbb{Z}/r}}$ is a dissolver of $G$.
\end{Thm}
\begin{proof}
Let $(\Xi,g,\Theta)$ be a constellation of $G$ and  let $u,v\in F$ be such that $\pi_1^G(u)\subseteq \Xi$, $\pi_1^G(v)\subseteq \Theta$ and $[u]_G=g=[v]_G$. Our goal is to show that $[u]_{H^{\widetilde{\mathbb{Z}/r}}}\ne [v]_{H^{\widetilde{\mathbb{Z}/r}}}$, that is, $[uv^{-1}]_{H^{\widetilde{\mathbb{Z}/r}}}\ne 1$. If $[u]_H\ne[v]_H$ then we are done; so let us assume that $[u]_H=[v]_H$.

Let us denote the canonical morphism $H\twoheadrightarrow G$ by $\vp$ and set $L:=\ker\vp$; we note that $L=\psi^{-1}\big(\ker(G^{\widetilde{\mathbb{Z}/p}}\twoheadrightarrow G)\big)$ where $\psi$ is the canonical morphism $H=(G^{\widetilde{\mathbb{Z}/p}})^{\widetilde{\mathbb{Z}/q}}\twoheadrightarrow G^{\widetilde{\mathbb{Z}/p}}$. Choose an edge $f\in (\partial_\Theta^+\cup \partial_\Theta^-)^{\pm 1}$  and a word $w\in F$ which labels a closed path at $1$ in $\Gamma(G)$ which traverses the (geometric) edge $f$ exactly once. Let the corresponding path be $sft$, that is, $\pi_1^G(w)=sft$, and let $w_s$ be the prefix of $w$ corresponding to $s$, that is, $\pi_1^G(w_s)=s$. Suppose that the label of $f$ is $b$ (for $b\in A\cup A^{-1}$), that is, $f=([w_s]_G,b)$. Let $\tilde f:=([w_s]_H,b)$, which is an edge of $\Gamma(H)$. By Lemma \ref{key lemma}, the graph $\Gamma(H)\setminus L\tilde{f}^{\pm 1}$ is disconnected, with the endpoints of $\tilde{f}$ lying in different connected components. From $\vp(\pi_1^H(u))=\pi_1^G(u)$, $\vp(L\tilde f)=f$ and $f\notin \pi_1^G(u)$ we conclude that $\pi_1^H(u)\cap L\tilde{ f}^{\pm 1}=\emptyset$, that is, $\pi_1^H(u)\subseteq \Gamma(H)\setminus L\tilde{f}^{\pm 1}$ and $\pi_1^H(u)$ is contained in the connected component of $1$ of the latter graph. Now consider the path $\pi_1^H(w)$, which is, except for the  edge $\tilde f$ (which is traversed exactly once), also contained in $\Gamma(H)\setminus L\tilde{f}^{\pm 1}$. Hence, the two endpoints $1$ and $n:=[w]_H$ of that latter path are in distinct connected components of the graph $\Gamma(H)\setminus L\tilde{f}^{\pm 1}$. Since $[w]_G=1$ we have $n\in \ker\vp=L$ whence $nL=L$. It follows that multiplication of $\Gamma(H)$ by $n$ on the left leaves invariant the graph $\Gamma(H)\setminus L\tilde{f}^{\pm 1} = n(\Gamma(H)\setminus L\tilde{f}^{\pm 1})$ and therefore shifts the graph $\pi_1^H(u)$ to the isomorphic copy $n\pi_1^H(u)=\pi_n^H(u)$ which is contained in the connected component of $n$ in $\Gamma(H)\setminus L\tilde{f}^{\pm 1}$. In particular, $\pi_1^H(u)\cap \pi_n^H(u)=\emptyset$ (meaning that the graphs spanned by these paths are disjoint).

By Corollary \ref{detecting edges} there exists an edge $e\in \vp^{-1}(\partial^+_\Xi\cup \partial^-_\Xi)$ for which $\pi_1^H(u)\not\equiv 0\ (\bmod\ r)$. Since $e\in \vp^{-1}(\partial^+_\Xi\cup \partial^-_\Xi)$ we have that $e\notin \pi_1^H(v)$ (because $(\partial^+_\Xi\cup \partial^-_\Xi)\cap \pi_1^G(v)=\emptyset$) and therefore $e\notin \pi_{[u]_H}(v^{-1})$ (since, as $[u]_H=[v]_H$, the paths $\pi_1^H(v)$ and $\pi_{[u]_H}(v^{-1})$ span the same graphs)  so that
$$\pi_1^H(uv^{-1})(e)=\pi_1^H(u)(e)\not\equiv 0\ (\bmod\ r).$$

The shifted edge $ne$ belongs to $\pi_n^H(u)$, which is disjoint with $\pi_1^H(u)$, whence $ne\notin \pi_1^H(u)$. Since $\vp(ne)=\vp(e)\in \partial^+_\Xi\cup \partial^-_\Xi$ we have $ne\notin \pi_1^H(v)$ and therefore $ne\notin \pi_{[u]_H}(v^{-1})$ (by the same argument as earlier for $e$). Combination of $ne\notin \pi_1^H(u)$ and $ne\notin  \pi_{[u]_H}(v^{-1})$ entails that  $\pi_1^H(uv^{-1})(ne)=0$. Altogether,
$$\pi_1^H(uv^{-1})(e)\not\equiv \pi_1^H(uv^{-1})(ne)\ (\bmod\ r).$$
But $e$ and $ne$ have the same label. So from Corollary \ref{word problem} it follows that $[uv^{-1}]_{H^{\widetilde{\mathbb{Z}/r}}}\ne1$.
\end{proof}

As a consequence we get:
\begin{Thm} \label{tree-like} Let $G$ be a group and $(p_n)$ be a sequence of primes. Set $G_0:=G$ and $G_{n}:=G_{n-1}^{\widetilde{\mathbb{Z}/p_{n}}}$. Then the Cayley graph of the profinite group $\mathcal G:=\varprojlim G_n$ is tree-like.
\end{Thm}

\subsubsection{Distinct primes}
We work towards a proof that in case the primes in the sequence are pairwise distinct and none of them is a divisor of $\vert G\vert$ then the $PQ$-category generated by the inverse sequence $(G_n)$ is freely indexed but not locally extensible. Let $G$ be a group and $p$ be a prime which does not divide $\vert G\vert$. Let $N$ be the additive group of the algebra $\mathbb{F}_p[G]$; $G$ acts on $\mathbb{F}_p[G]$ by left multiplication ${}^h\sum\alpha(g)g=\sum\alpha(g)hg$; consider the semidirect product $K:=N\rtimes G$ subject to that action.  There exists a morphism from $N\rtimes G$ to the additive cyclic group of $\mathbb{F}_p$ of order $p$ which maps the element $(\sum_{g\in G}g,1)$ to the generating element $1$, namely:
\begin{equation}\label{average}
(\alpha,k)\mapsto \frac{1}{\vert G\vert}\sum_{g\in G}\alpha(g).\end{equation}
 Similarly but more easily as in Proposition \ref{center of Gaschuetz} one can show that the center $Z$ of $N\rtimes G$ (is contained in $N$) and consists of the cyclic group generated by the element $\sum_{g\in G}g$. Moreover, $G$ naturally acts on $N/Z$ and $(N\rtimes G)/Z\cong (N/Z)\rtimes G$ (slightly abusing notation and identifying elements of the form $(\alpha,1)$ with $\alpha$).  We intend to describe $[K,K]\cap N$. For each $k\in G, k\ne 1$, the element
\[
\begin{aligned}
(1-k,1)&=(1,1)(0,k)(-1,1)(0,k^{-1})\\
    &=(1,1)(0,k)(1,1)^{-1}(0,k)^{-1}
\end{aligned}
\]
is in $[K,K]\cap N$. From  (\ref{average}), $(\sum_{g\in G}g,1)\notin [K,K]$.  Since the set
$$\big\{ 1-k\mid k\in G,k\ne 1\big\}\cup\big\{\sum_{g\in G}g\big\}$$
forms a basis of $N$ (considered as an $\mathbb{F}_p$-vector space) it follows that $$\{1-k+Z\mid k\in G,k\ne 1\}$$ forms a basis of $N/Z$ so that $N/Z\subseteq [K/Z,K/Z]$.  As a consequence, the canonical map $K/Z\twoheadrightarrow (K/Z)^{ab}$ from $K/Z$ to its abelianization factors through $G$. We have thus obtained the following
\begin{Cor} \label{abelian quotients} Every abelian quotient of $(N/Z)\rtimes G$ is a quotient of $G$.
\end{Cor}
The connection of $\mathbb{F}_p[G]\rtimes G$ with the Gasch\"utz $p$-extension becomes clear from the following result of Gasch\"utz:
\begin{Thm} \cite[Satz 4.]{Gaschuetz}\label{Gaschuetz} For every $A$-generated group $G$ and every prime $p$ not dividing $\vert G\vert$:
$$G^{\mathbb{Z}/p}\cong \big(\mathbb{F}_p\oplus (\bigoplus_{i=1}^{\vert A\vert-1}\mathbb{F}_p[G])\big)\rtimes G$$
where $G$ acts trivially on  $\mathbb{F}_p$ and by left multiplication on each copy of $\mathbb{F}_p[G]$.
\end{Thm}
In this model of $G^{\mathbb{Z}/p}$, the center is easily identified as:
$$\mathbb{F}_p\oplus (\bigoplus_{i=1}^{\vert A\vert-1} Z )$$
where, as above, $Z$ is the cyclic subgroup of $\mathbb{F}_p[G]$ generated by $\sum_{g\in G}g$. An immediate consequence is a similar model of the group $G^{\widetilde{\mathbb{Z}/p}}$:
\begin{Cor} For every group $G$ and every prime $p$ not dividing $\vert G\vert$:
$$G^{\widetilde{\mathbb{Z}/p}}\cong (\bigoplus_{i=1}^{\vert A\vert -1}\mathbb{F}_p[G]/Z)\rtimes G.$$
\end{Cor}
In particular, $G^{\widetilde{\mathbb{Z}/p}}$ is a subdirect power of $(\mathbb{F}_p[G]/Z)\rtimes G$.
The arguments which lead to Corollary \ref{abelian quotients} can therefore be applied componentwise to the latter model of $G^{\widetilde{\mathbb{Z}/p}}$ to obtain that the canonical map $G^{\widetilde{\mathbb{Z}/p}}\twoheadrightarrow (G^{\widetilde{\mathbb{Z}/p}})^{ab}$  factors through $G$.
\begin{Cor} \label{abelian quotients 2} Every abelian quotient of $G^{\widetilde{\mathbb{Z}/p}}$ is a quotient of $G$.
\end{Cor}
We are ready for the second main result.
\begin{Thm} \label{not locally extensible} Let $G:=G_0$ be a group and $(p_n)_{n\ge 1}$ be a sequence of distinct primes none of which divides $\vert G\vert$. For $n\ge 1 $ set $G_n:=G_{n-1}^{\widetilde{\mathbb{Z}/{p_n}}}$; then the $PQ$-category $\mathfrak{N}$ generated by the sequence $(G_n)$ is not locally extensible.
\end{Thm}
\begin{proof} We show that for no prime $p$ and no $n\in \mathbb{N}$ the morphism $G_n\twoheadrightarrow G_0$ factors through $G_0^{\mathbb{Z}/p}$. Suppose the contrary is true: then for some prime $p$ and some $n$, ${G}_n\twoheadrightarrow G_0^{\mathbb{Z}/p}\twoheadrightarrow G_0$. Since no prime divisor of $\vert\ker(G_n\twoheadrightarrow G_0)\vert$ divides $\vert G_0\vert$ it follows that $p$ cannot divide $\vert G_0 \vert$ whence $p\in \{p_1\dots,p_n\}$. But the cyclic group of $\mathbb{F}_p$ of order $p$  is a quotient of $G_0^{\mathbb{Z}/p}$ while from Corollary \ref{abelian quotients 2} and by induction it follows that the only abelian quotients of $G_n$ are those of $G_0$. However, the latter does not have the cyclic group of order $p$ among its quotients since $p\nmid \vert G_0\vert$.
\end{proof}
Finally we show that the group $\mathcal{G}=\varprojlim G_n$ (the sequence $(G_n)$ chosen as above) is freely indexed. Let $t:=\vert A\vert\ge 2$ and  for $n\in \mathbb{N}$ set $\mathcal{U}_n:=\ker(\mathcal{G}\twoheadrightarrow G_n)$. We intend to show that $d(\mathcal{U}_n)=(t-1)\vert G_n\vert +1$ (and only the inequality $\ge$ has to be proved). Since $\ker(G_{n+2}\twoheadrightarrow G_n)$ is a quotient of $\mathcal{U}_n$ it suffices to show that the rank of the latter is not smaller than $(t-1)\vert G_n\vert +1$. So, for $\ell\ge 1$ and a fixed $n$ let $K_\ell:=\ker(G_{n+\ell} \twoheadrightarrow G_n)$. We have
$$G_{n+1}=(\bigoplus_{i=1}^{t-1} \mathbb{F}_{p_{n+1}}[G_n]/Z_n)\rtimes G_n,$$
where $Z_n$ is the cyclic subgroup of $\mathbb{F}_{p_{n+1}}[G_n]$ generated by $\sum_{g\in G_n}g$. Hence
$$K_1=\bigoplus_{i=1}^{t-1} \mathbb{F}_{p_{n+1}}[G_n]/Z_n.$$
Moreover,
\[
\begin{aligned}
G_{n+2} &= (\bigoplus_{i=1}^{t-1}\mathbb{F}_{p_{n+2}}[G_{n+1}]/Z_{n+1})\rtimes G_{n+1} \\
        &= (\bigoplus_{i=1}^{t-1}\mathbb{F}_{p_{n+2}}[G_{n+1}]/Z_{n+1}) \rtimes \big((\bigoplus_{i=1}^{t-1}\mathbb{F}_{p_{n+1}}[G_{n}]/Z_{n})\rtimes G_n\big),
\end{aligned}
\]
where $Z_{n+1}$ is the cyclic subgroup of $\mathbb{F}_{p_{n+2}}[G_{n+1}]$ generated by $\sum_{g\in G_{n+1}}g$,
 and thus
$$K_2=(\bigoplus_{i=1}^{t-1}\mathbb{F}_{p_{n+2}}[G_{n+1}]/Z_{n+1}) \rtimes \big((\bigoplus_{i=1}^{t-1}\mathbb{F}_{p_{n+1}}[G_{n}]/Z_{n})\rtimes \{1\}\big).
$$
Set $r:=\vert G_n\vert$ and let $q:=p_{n+1}^{(r-1)(t-1)}=\vert K_1\vert$; then $\vert G_{n+1}\vert =q\cdot r$. For
$$L:=\bigoplus_{i=1}^{t-1}\mathbb{F}_{p_{n+2}}[G_{n+1}]/Z_{n+1}$$
we  have $$d(L)=(\vert G_{n+1}\vert -1)(t-1)=(q\cdot r-1)(t-1).$$
Now  $L\unlhd K_2$ and $[K_2:L]=\vert K_1\vert =q$; since the quantity in the Schreier formula is an upper bound for the rank of the subgroup we have
$$d(L)\le (d(K_2)-1)\vert K_1\vert +1=(d(K_2)-1)q+1.$$
Hence
\[
\begin{aligned}
d(K_2) &\ge \frac{d(L)-1}{q}+1\\
        &= \frac{(qr-1)(t-1)-1}{q}+1\\
        &=(r-\frac{1}{q})(t-1)-\frac{1}{q} +1\\
        &=rt-\frac{t}{q}-r+\frac{1}{q}-\frac{1}{q}+1 =r(t-1)+1-\frac{t}{q}.\\
\end{aligned}
\]
 Since $q=p_{n+1}^{(r-1)(t-1)}\ge 2^{3(t-1)}>t$ for all $t\ge 2$ we have that $\frac{t}{q}< 1$. Since $d(K_2)$ is an integer we get $d(K_2)\ge r(t-1)+1$ for $r=\vert G_n\vert$. Since $\{\mathcal{U}_n\mid n\in \mathbb{N}\}$ forms a basis of the neighborhoods of $1$ in $\mathcal G$ the next result is a consequence of \cite[Lemma 2.5 (ii)]{LubotzkyDries}:
\begin{Thm}\label{freely indexed} If all primes $p_n$ are distinct and none of them divides $\vert G_0\vert $ then $\mathcal{G}$ is freely indexed.
\end{Thm}

\subsection{Formations} We are going to construct a formation of solvable groups which is arboreous and freely indexed (therefore Hall) but not locally extensible. Let $(p_n)_{n\ge 0}$ be a sequence of pairwise distinct primes. For every $t \ge 2$ we fix an alphabet $A_t$ of size $t$ and set $G_{0t}:=(\mathbb{Z}/{p_0}\mathbb{Z})^{A_t}$ and by induction, for $n\ge 1$: $G_{nt}:=G_{n-1,t}^{\widetilde{\mathbb{Z}/{p_{n}}}}$ and set
\begin{equation}\label{free pro-F}  \mathcal{G}_t:=\varprojlim_{n\ge 0} G_{nt}.
\end{equation}
According to Theorems \ref{tree-like}, \ref{freely indexed} and \ref{not locally extensible} all $\mathcal{G}_t$ have tree-like Cayley graphs and are freely indexed. We intend to show that the class of all finite quotients of all $\mathcal{G}_t$ forms a formation with respect to which the group $\mathcal{G}_t$ is the $A_t$-generated free profinite object.

First we to show that the class of all quotients of
$$\{G_{nt}\mid n\ge 0, t\ge 2\}$$
is a formation. An essential step to this goal is to show that the $G_{nt}$ are ``relatively free'', that is, they satisfy a certain universal mapping property. We shall use the following well known fact.
\begin{Lemma} \label{factoring center}
Let $G\overset{\vp}{\twoheadrightarrow}H$ be groups with respective centers $Z(G)$ and $Z(H)$. Then $\vp(Z(G))\subseteq Z(H)$, hence $\vp$ induces a canonical morphism $\ol{\vp}:G/Z(G)\twoheadrightarrow H/Z(H)$.
\end{Lemma}
Throughout we let the groups $G_{ns}$ be equipped with the standard generating set $A_s$, though, at some point we will have to consider also non-standard generators.
\begin{Thm}\label{universal property}
Let $t\ge s$ and $n\ge 0$; then every map $\beta:A_t\to G_{ns}$ such that $\beta(A_t)$ generates $G_{ns}$ can be extended to a morphism $\widehat{\beta}:G_{nt}\twoheadrightarrow G_{ns}$.
\end{Thm}
\begin{proof} Note that for the special case $t=s$ this means that the groups $G_{ns}$ are \emph{homogeneous} in the sense of Gasch\"utz \cite{Ga:lifting generators}. The proof is, for fixed $t\ge s$ by induction on $n$. For $n=0$ the claim says that every map $\beta:A_t\to (\mathbb{Z}/p_0\mathbb{Z})^s$ such that $\beta(A_t)$ generates $(\mathbb{Z}/p_0\mathbb{Z})^s$ can be extended to a morphism $(\mathbb{Z}/p_0\mathbb{Z})^t\twoheadrightarrow(\mathbb{Z}/p_0\mathbb{Z})^s$, which is obviously correct since $A_t$ is a basis of the vector space $(\mathbb{Z}/p_0\mathbb{Z})^t$ and $t\ge s$.

For the following arguments, the reader may check with Figure 4. We assume that $A_t$ is a subset (standard generating set) of $$G_{nt}, G_{nt}^{\widetilde{\mathbb{Z}/p_{n+1}}}, G_{nt}^{\mathbb{Z}/p_{n+1}}, G_{ns}^{A_t,\mathbb{Z}/p_{n+1}}$$ while $A_s$ is a subset (standard generating set) of $G_{ns}, G_{ns}^{\widetilde{\mathbb{Z}/p_{n+1}}}, G_{ns}^{\mathbb{Z}/p_{n+1}}$; but $A_t$ is embedded (not necessarily injectively) in $G_{ns}^{\widetilde{\mathbb{Z}/p_{n+1}}}, G_{ns}^{\mathbb{Z}/p_{n+1}}$ via $\beta$ and $\beta_1$.
\begin{figure}[ht]
\begin{tikzpicture} [xscale=1.8, yscale=1.6]
\node at (2.8,1.9) {$G_{ns}$};
\node at (2.8,3.9) {$G_{nt}$};
\node at (1,2) {$G_{ns}^{\widetilde{\mathbb{Z}/p_{n+1}}}$};
\node at (1,4) {$G_{nt}^{\widetilde{\mathbb{Z}/p_{n+1}}}$};
\node at (-1,2)   {$G_{ns}^{{\mathbb{Z}/p_{n+1}}}$};
\node at (-1,4) {$G_{nt}^{{\mathbb{Z}/p_{n+1}}}$};
\node at (-3,2) {$G_{ns}^{{A_t,\mathbb{Z}/p_{n+1}}}$};
\node at (0,0) {$A_t$};
\draw[->>] (-0.5,1.9)--(0.5,1.9);
\draw[->>] (1.5,1.9)--(2.5,1.9);
\draw[->>] (-2.5,1.9)--(-1.5,1.9);
\draw[->>] (-0.5,3.9)--(0.5,3.9);
\draw[->>] (1.5,3.9)--(2.5,3.9);
\draw[->] (-0.1,0.2)-- (-0.9,1.8);
\draw[->] (0.1,0.2)--(0.77,1.69);
\draw[->>] (0.95,3.7)--(0.95,2.35);
\draw[->>] (-1,3.7)-- (-1,2.3);
\draw[->>] (2.8,3.7)--(2.8,2.2);
\draw[->>] (-1.35,3.7)--(-3,2.35);
\node[right] at (2.8,3) {$\gamma$};
\node at (2,2) {$\varphi_s$};
\node at (0,2) {$\tau_s$};
\node at (-2,2.05) {$\widehat{\beta_1}$};
\node at (0,4) {$\tau_t$};
\node[right] at (-1,3) {$\alpha$};
\node at ((0.95,3) {$\widetilde{\alpha}=\widehat{\beta}$};
\node[left] at (-0.5,1) {$\beta_1$};
\node[right] at (0.5,1) {$\beta$};
\node[above left] at (-2,3) {$\widehat{\gamma}$};
\end{tikzpicture}
\caption{}
\end{figure}
So, let $n\ge 0$ and suppose the claim be true for $n$. Let $\beta:A_t\to G_{n+1,s}=G_{ns}^{\widetilde{\mathbb{Z}/p_{n+1}}}$ be a mapping such that $\beta(A_t)$ generates $G_{ns}^{\widetilde{\mathbb{Z}/p_{n+1}}}$. Then $\vp_s(\beta(A_t))$ generates $G_{ns}$ ($\vp_s$  being the canonical morphism $G_{n,s}^{\widetilde{\mathbb{Z}/p_{n+1}}}\twoheadrightarrow G_{ns}$). By the induction assumption there exists a morphism $\gamma: G_{nt}\twoheadrightarrow G_{ns}$ such that $\gamma(a)=(\vp_s(\beta(a))$ for all $a\in A_t$. Since the  Gasch\"utz extension is functorial (Proposition \ref{Gaschuetz functorial}), there exists a (canonical) morphism $\widehat{\gamma}: G_{nt}^{\mathbb{Z}/p_{n+1}}\twoheadrightarrow G_{ns}^{A_t,\mathbb{Z}/p_{n+1}}$ (that is, $\wh{\gamma}(a)=a$ for all $a\in A_t$ and $A_t$ is the standard generating set of $G_{ns}^{A_t,\mathbb{Z}/p_{n+1}}$: $G_{ns}^{A_t,\mathbb{Z}/p_{n+1}}$ is the Gasch\"utz $p_{n+1}$-extension of $G_{ns}$ but with respect to the generating set $\vp_s(\beta(A_t))$). From a result of Gasch\"utz \cite{Ga:lifting generators}, the generating set $\beta(A_t)$ in $G_{ns}^{\widetilde{\mathbb{Z}/p_{n+1}}}$ can be lifted to a generating set of $G_{ns}^{\mathbb{Z}/p_{n+1}}$, hence there exists a mapping $\beta_1:A_t\to G_{ns}^{\mathbb{Z}/p_{n+1}}$ such that $\beta_1(A_t)$ generates $G_{ns}^{\mathbb{Z}/p_{n+1}}$ and $\beta=\tau_s\circ \beta_1$ for the canonical morphism $\tau_s:G_{ns}^{\mathbb{Z}/p_{n+1}}\twoheadrightarrow G_{ns}^{\widetilde{\mathbb{Z}/p_{n+1}}}$ (note that $\tau_s$ is canonical with respect to $A_s$ by definition, but respects also the embedded generating set $A_t$). From the universal property of the Gasch\"utz extension the mapping $\beta_1$ extends to a morphism $\wh{\beta_1}:G_{ns}^{A_t,\mathbb{Z}/p_{n+1}} \twoheadrightarrow G_{ns}^{\mathbb{Z}/p_{n+1}}$, that is, $\wh{\beta_1}(a)=\beta_1(a)$ for all $a\in A_t$. In order to avoid confusion: $G_{ns}^{A_t,\mathbb{Z}/p_{n+1}}$ is the Gasch\"utz $p_{n+1}$-extension of $G_{ns}$ but with respect to the generating set $\vp_s(\beta(A_t))$; the group $G_{ns}^{\mathbb{Z}/p_{n+1}}$ is \textbf{some} $A_t$-generated extension of an elementary abelian $p_{n+1}$-group by $G_{ns}$ which maps to $G_{ns}$ via the $A_t$-generators respecting morphism $\vp_s\circ \tau_s$; hence the $A_t$-canonical mapping $G_{ns}^{A_t,\mathbb{Z}/p_{n+1}}\twoheadrightarrow G_{ns}$ factors through $G_{ns}^{\mathbb{Z}/p_{n+1}}$ via the $A_t$-generators respecting morphism $\wh{\beta}_1$.
The morphism
$$ \wh{\beta_1}\circ\wh{\gamma}:G_{nt}^{\mathbb{Z}/p_{n+1}}\twoheadrightarrow G_{ns}^{\mathbb{Z}/p_{n+1}}$$ then satisfies
$$(\wh{\beta_1}\circ\wh{\gamma})(a)=\wh{\beta_1}(\wh{\gamma}(a))=\wh{\beta_1}(a)=\beta_1(a)$$
for all $a\in A_t$. Set $\alpha:=\wh{\beta_1}\circ\wh{\gamma}$ (then $\alpha(a)=\beta_1(a)$ for all $a\in A_t$). The latter morphism, by Lemma \ref{factoring center}, induces a morphism $\widetilde{\alpha}:G_{nt}^{\widetilde{\mathbb{Z}/p_{n+1}}}\twoheadrightarrow G_{ns}^{\widetilde{\mathbb{Z}/p_{n+1}}}$ such that $$\widetilde{\alpha}\circ\tau_t=\tau_s\circ \alpha$$
where $\tau_t:G_{nt}^{\mathbb{Z}/p_{n+1}}\twoheadrightarrow G_{nt}^{\widetilde{\mathbb{Z}/p_{n+1}}}$ is the canonical morphism. Setting $\wh{\beta}:=\widetilde{\alpha}$ then we have $$\wh{\beta}(a)=\widetilde{\alpha}(a)=(\widetilde{\alpha}\circ\tau_t)(a)=(\tau_s\circ \alpha)(a)=(\tau_s\circ \beta_1)(a)=\beta(a),$$ as required.
\end{proof}

An immediate consequence is this:
\begin{Cor}\label{general universal property} Let $t\ge s$ and $G$ be a morphic image of $G_{ns}$. Then for every map $\beta: A_t\to G$ such that $\beta(A_t)$ generates $G$ there exists a  morphism $\wh{\beta}:G_{nt}\to G$ such that $\wh{\beta}(a)=\beta(a)$ for all $a\in A_t$.
\end{Cor}
\begin{proof} By  {\emph{lifting the generators}} \cite{Ga:lifting generators} there exists a map $\beta_1:A_t\to G_{nt}$ such that $\beta_1(A_t)$ generates $G_{nt}$ and $\vp\circ \beta_1=\beta$ where $\vp:G_{nt}\twoheadrightarrow G$ is the canonical morphism. By Theorem \ref{universal property} there exists a morphism $\wh{\beta}:G_{nt}\twoheadrightarrow G_{ns}$ such that $\wh{\beta}(a)=\beta_1(a)$ for all $A_t$. The morphism $\vp\circ \wh{\beta}:G_{nt}\twoheadrightarrow G$ then has the required property.
\end{proof}
We arrive at the first main result of the present subsection.
\begin{Thm}\label{quotients form formation} The class
$$\mathfrak{F}:=\bm{\mathsf{Q}}\{G_{nt}\mid n\ge 0,t\ge 2\}$$ 
of all quotients of all groups $G_{nt}$ is a formation.
\end{Thm}
\begin{proof}
The class is obviously closed under quotients. We need to prove that it is closed under subdirect products. Let $G_1,\dots, G_k$ be members of $\mathfrak{F}$ and choose a number $n$ such that every $G_i$ is a quotient of $G_{nt_i}$ for some $t_i$. In particular, $G_i$ is $t_i$-generated. Let $G\subseteq \prod_{i=1}^k G_i$ be a subdirect product. Choose a generating tuple $T$ of $G$ of size $t$, say, and we may assume that $t\ge t_i$ for all $i$. Let $\alpha:A_t\to T$ be a bijection  and for every $i$ denote by $\pi_i:G\twoheadrightarrow G_i$ the projection. Since $\pi_i(T)$ generates $G_i$, by Corollary \ref{general universal property} there is a morphism $\vp_i:G_{nt}\twoheadrightarrow G_i$ extending the mapping $\pi_i\circ \alpha:A_t\to G_i$. It follows that the morphism $\vp_1\times \cdots\times \vp_k:G_{nt}\to \prod_{i=1}^k G_i$, $g\mapsto (\vp_1(g),\dots,\vp_k(g))$ maps onto $G$, as requested.
\end{proof}

In order to show that for every $t\ge 2$ the group $\mathcal{G}_t$ is the $A_t$-generated free pro-$\mathfrak{F}$ group it
suffices to prove that  every $A_t$-generated member of $\mathfrak{F}$ is a quotient of $G_{nt}$ for some $n\ge 0$. For a prime $p$ let $\mathbf{Ab}_p$ be the variety  of all finite elementary abelian $p$-groups. We set $$\mathbf{F}_0:=\mathbf{Ab}_{p_0}\text{ and }\mathbf{F}_n:= \mathbf{Ab}_{p_n}*\mathbf{F}_{n-1}\text{ for }n>0$$ where, by definition, the latter denotes the class of all finite groups which are extensions of elementary abelian $p_n$-groups by groups in $\mathbf{F}_{n-1}$. The class $\bigcup_{n\ge 0} \mathbf{F}_n$ is a variety of finite groups  and by definition, $\mathfrak{F}\subseteq \mathbf{F}$. We show that every $A_t$-generated member of $\mathfrak{F}$ is a quotient of some $G_{nt}$ by showing by induction on $n$ that this claim is true for all members of $\mathfrak{F}_n:=\mathfrak{F}\cap \mathbf{F}_n$. The claim for $n=0$ is obviously true: every $A_t$-generated elementary abelian $p_0$-group is a quotient of $G_{0t}=(\mathbb{Z}/p_0\mathbb{Z})^t$. So, let $n>0$ and suppose the claim be true for $n-1$ and let $G$ be an $A_t$-generated member of $\mathfrak{F}_n$; there exists a normal subgroup $N\unlhd G$ which is in $\mathbf{Ab}_{p_n}$ such that the quotient $G/N$ is in $\mathfrak{F}_{n-1}$. By the induction hypothesis, $G/N$ is a quotient of $G_{n-1,t}$. By the universal property of the Gasch\"utz extension, $G$ is a quotient of $(G/N)^{\mathbb{Z}/p_n}$, that is,  $(G/N)^{\mathbb{Z}/p_n}\twoheadrightarrow G$ and the kernel $K$ of the latter morphism is an elementary abelian $p_n$-group. Now by Theorem \ref{Gaschuetz},
$$(G/N)^{\mathbb{Z}/p_n}\cong \big(\mathbb{F}_{p_n}\oplus(\bigoplus_{i=1}^{t-1} \mathbb{F}_{p_n}[G/N])\big)\rtimes (G/N)$$
and the kernel $K$ is contained in $\mathbb{F}_{p_n}\oplus(\bigoplus_{i=1}^{t-1} \mathbb{F}_{p_n}[G/N])$. Let $Z$ be the cyclic subgroup  of  $\mathbb{F}_{p_n}[G/N])$ generated by $\sum_{x\in G/N}x$. We claim that
$$\mathbb{F}_{p_n}\oplus(\bigoplus_{i=1}^{t-1}Z)\subseteq K.$$
If this were not the case then $\big(\mathbb{F}_{p_n}\oplus(\bigoplus_{i=1}^{t-1} \mathbb{F}_{p_n}[G/N])\big)\rtimes (G/N)$ would have a cyclic image of order $p_n$ (either via the direct summand $\mathbb{F}_{p_n}$ or via (\ref{average}) applied to one of the  summands whose $Z$ is not in $K$). However, from Corollary \ref{abelian quotients 2} it follows that the only abelian members of $\mathfrak{F}$ are elementary abelian $p_0$-groups, leading to a contradiction. This means that $K$ contains the center of $(G/N)^{\mathbb{Z}/p_n}$, so that $G=(G/N)^{\mathbb{Z}/p_n}/K$ is a quotient of $(G/N)^{\widetilde{\mathbb{Z}/p_n}}$. Finally, from Lemma \ref{factoring center} in combination with Proposition \ref{Gaschuetz functorial},  $(G/N)^{\widetilde{\mathbb{Z}/p_n}}$
is a quotient of $G_{nt}=G_{n-1,t}^{\widetilde{\mathbb{Z}/p_n}}$. Altogether we have proved:
\begin{Thm} \begin{enumerate}
\item For every $t\ge 2$, the $A_t$-generated members of $\mathfrak{F}$ are exactly the quotients of $\{G_{nt}\mid n\ge 0\}$.
\item For all $t\ge 2$ the $A_t$-generated free pro-$\mathfrak{F}$ groups are the groups $\mathcal{G}_t$ of  (\ref{free pro-F}).
\end{enumerate}
\end{Thm}
\begin{Cor} The formation $\mathfrak{F}$ is freely indexed and arboreous, therefore Hall, but not locally extensible.
\end{Cor}

\subsection*{Acknowledgment}\label{sec:Acknowledgments} The authors would like to thank the referees for their very careful (and quick) reading of the paper. Their comments have led to considerable improvements.

\end{document}